\documentclass[12pt,CJK]{article}
\usepackage{amsfonts,amsmath,amssymb,amsthm,mathrsfs}
\usepackage{latexsym}
\usepackage{graphicx}
\usepackage{indentfirst}
\usepackage{color}
\usepackage{fancyhdr}
\usepackage[colorlinks,linktocpage,linkcolor=blue]{hyperref}

\theoremstyle{plain}
\newtheorem{thm}{Theorem}[section]
\newtheorem{definition}{Definition}
\newtheorem{lemma}[thm]{Lemma}
\newtheorem{corollary}[thm]{Corollary}
\newtheorem{remark}[thm]{Remark}
\numberwithin{equation}{section}

\theoremstyle{remark}

\def\Xint#1{\mathchoice
  {\XXint\displaystyle\textstyle{#1}}%
  {\XXint\textstyle\scriptstyle{#1}}%
  {\XXint\scriptstyle\scriptscriptstyle{#1}}%
  {\XXint\scriptscriptstyle\scriptscriptstyle{#1}}%
  \!\int}
\def\XXint#1#2#3{{\setbox0=\hbox{$#1{#2#3}{\int}$}
  \vcenter{\hbox{$#2#3$}}\kern-.5\wd0}}

\def\dashint{\Xint-}

\begin{document}
\allowdisplaybreaks
\pagestyle{myheadings}\markboth{$~$ \hfill {\rm Q. Xu,} \hfill $~$} {$~$ \hfill {\rm  } \hfill$~$}
\author{Qiang Xu
\thanks{Corresponding author.}
\thanks{Email: xuqiang@math.pku.edu.cn.}
\quad Shulin Zhou
\thanks{Email: szhou@math.pku.edu.cn.} \\
School of Mathematical Sciences, Peking University, \\
Beijing, 100871, PR China. \vspace{0.5cm}
}

%


\title{\textbf{$L^p$ Neumann Problems in\\
Homogenization of General Elliptic Systems} }
\maketitle
\begin{abstract}
In this paper,
we extend the nontangential maximal function estimate obtained by C. Kenig, F. Lin and Z. Shen in
\cite{KFS1} to the nonhomogeneous elliptic operators with rapidly oscillating periodic coefficients.
The result relies on the previous work \cite{X2}, and optimal boundary estimates which is based upon
certain estimates on convergence rates.
Compared to the homogeneous case, the additional bootstrap process seems inevitable, and
the Neumann boundary corrector caused by the lower order term are still useful here.
\end{abstract}

\section{Instruction and main results}

The main purpose of this paper is to investigate
the well-posedness of $L^p$ Neumann problems
for nonhomogeneous elliptic systems, arising in the homogenization theory.
More precisely, we continue to consider the following operators depending on a parameter $\varepsilon > 0$,
\begin{eqnarray*}
\mathcal{L}_{\varepsilon} =
-\text{div}\big[A(x/\varepsilon)\nabla +V(x/\varepsilon)\big] + B(x/\varepsilon)\nabla +c(x/\varepsilon) +\lambda I
\end{eqnarray*}
where $\lambda\geq 0$ is a constant, and $I$ is an identity matrix.

Let $d\geq 3$, $m\geq 1$, and $1 \leq i,j \leq d$ and $1\leq \alpha,\beta\leq m$.
Suppose that $A = (a_{ij}^{\alpha\beta})$, $V=(V_i^{\alpha\beta})$, $B=(B_i^{\alpha\beta})$, $c=(c^{\alpha\beta})$ are real measurable functions,
satisfying the following conditions:
\begin{itemize}
\item the uniform ellipticity condition
\begin{equation}\label{a:1}
 \mu |\xi|^2 \leq a_{ij}^{\alpha\beta}(y)\xi_i^\alpha\xi_j^\beta\leq \mu^{-1} |\xi|^2
 \quad \text{for}~y\in\mathbb{R}^d,~\text{and}~\xi=(\xi_i^\alpha)\in \mathbb{R}^{md},~\text{where}~ \mu>0;
\end{equation}
 (The summation convention for repeated indices is used throughout.)
\item the periodicity condition
\begin{equation}\label{a:2}
A(y+z) = A(y),~~ V(y+z) = V(y),
~~B(y+z) = B(y),~~ c(y+z) = c(y)
\end{equation}
for $y\in\mathbb{R}^d$ and $z\in \mathbb{Z}^d$;
\item the boundedness condition
\begin{equation}\label{a:3}
 \max\big\{\|V\|_{L^{\infty}(\mathbb{R}^d)},
 ~\|B\|_{L^{\infty}(\mathbb{R}^d)},~\|c\|_{L^{\infty}(\mathbb{R}^d)}\big\}
 \leq \kappa;
\end{equation}
\item the regularity condition
\begin{equation}\label{a:4}
 \max\big\{ \|A\|_{C^{0,\tau}(\mathbb{R}^d)},~ \|V\|_{C^{0,\tau}(\mathbb{R}^d)},
 ~\|B\|_{C^{0,\tau}(\mathbb{R}^d)}\big\} \leq \kappa,
 \qquad \text{where}~\tau\in(0,1)~\text{and}~\kappa > 0.
\end{equation}
\end{itemize}

Although we do not seek the operator $\mathcal{L}_\varepsilon$ to be a self-adjoint operator,
the symmetry condition on its leading term, i.e.,
\begin{equation*}
A^* = A \quad (a_{ij}^{\alpha\beta} = a_{ji}^{\alpha\beta}\big)
\end{equation*}
is necessary in the later discussion. To ensure the solvability,
the following constant is crucial,
\begin{equation*}
\lambda_0 = \frac{c(m,d)}{\mu}\Big\{\|V\|_{L^\infty(\mathbb{R}^d)}^2 +
\|B\|_{L^\infty(\mathbb{R}^d)}^2+ \|c\|_{L^\infty(\mathbb{R}^d)}\Big\}.
\end{equation*}

Throughout the paper, we always assume $\Omega\subset\mathbb{R}^d$ is a bounded Lipschitz domain, and $r_0$ denotes
the diameter of $\Omega$, unless otherwise stated.
In order to state the Neumann boundary value problem,
the conormal derivatives related to $\mathcal{L}_\varepsilon$ is defined as
\begin{equation*}
\frac{\partial}{\partial\nu_\varepsilon}
 = n\cdot\big[A(\cdot/\varepsilon)\nabla + V(\cdot/\varepsilon)\big]
 \qquad \text{on}~\partial\Omega,
\end{equation*}
where $n=(n_1,\cdots,n_d)$ is the outward unit normal
vector to $\partial\Omega$.

\begin{thm}[nontangential maximal function estimates]\label{thm:1.1}
Let $1<p<\infty$.
Suppose that the coefficients $\eqref{a:1}$, $\eqref{a:2}$, $\eqref{a:3}$ and
$\eqref{a:4}$ with $\lambda\geq\max\{\mu,\lambda_0\}$ and $A^*=A$. Let $\Omega$
be a bounded $C^{1,\eta}$ domain with some $\eta\in(0,1)$.
Then for any $g\in L^p(\partial\Omega;\mathbb{R}^m)$,
the weak solution $u_\varepsilon\in H^1(\Omega;\mathbb{R}^m)$ to
\begin{equation}\label{pde:1}
(\mathbf{NH_\varepsilon})\left\{
\begin{aligned}
\mathcal{L}_\varepsilon(u_\varepsilon) &= 0 &\quad &\emph{in}~~\Omega, \\
 \frac{\partial u_\varepsilon}{\partial\nu_\varepsilon} &= g &\emph{n.t.~}&\emph{on} ~\partial\Omega, \\
 (\nabla u_\varepsilon)^* &\in L^p(\partial\Omega), &\quad &
\end{aligned}\right.
\end{equation}
satisfies the uniform estimate
\begin{equation}\label{pri:1.1}
\|(\nabla u_\varepsilon)^*\|_{L^p(\partial\Omega)}
+ \|(u_\varepsilon)^*\|_{L^p(\partial\Omega)} \leq C\|g\|_{L^p(\partial\Omega)}
\end{equation}
where $C$ depends on $\mu,\kappa,\tau,\lambda,m,d$ and $\Omega$.
\end{thm}

Note that the second line of $(\mathbf{NH_\varepsilon})$
means that the conormal derivative of $u_\varepsilon$ converges to $f$ in a nontangenial way
instead of in the sense of trace, and using the abbreviation ``n.t.'' depicts this difference.
The notation $(\nabla u_\varepsilon)^*$ in the third line
represents the nontangential maximal function of
$\nabla u_\varepsilon$ on $\partial\Omega$ (see Definition $\ref{def:1}$).

The main strategy in the proof of the above theorem has been well developed in \cite{KFS1}.
Roughly speaking, the proof should be divided into two parts: (i) $2\leq p<\infty$
and (ii) $1<p<2$. On account of a real method given by Z. Shen in \cite{S5},
originally inspired by L. Caffarelli and I. Peral in \cite{CP},
the case (i) will be reduced to a revise H\"older inequality. For the case (ii), one may
derive the estimate $\|(\nabla u_\varepsilon)^*\|_{L^1(\partial\Omega)}\leq
C\|g\|_{H^1_{at}(\partial\Omega)}$ as in \cite{KFS1,DK}, where the right-hand side means the given data
g is in the atomic $H^1$ space (see for example \cite[pp.438]{DK}), and then by a interpolating argument one may obtain the desired estimate.

However, to complete the whole proof of Theorem $\ref{thm:1.1}$ is not as easy as it appears.
In terms of layer potential methods, we first establish the estimate $\eqref{pri:1.1}$ for $p=2$
in Lipschitz domains (see \cite[Theorem 1.6]{X2}). Then, applying the real method
(see Lemma $\ref{lemma:2.1}$) to the nonhomogeneous operators, one may derive the following result.

\begin{thm}\label{thm:2.1}
Let $p>2$ and $\Omega$ be a bounded Lipschitz domain. Assume that
\begin{equation}\label{a:2.1}
\bigg(\dashint_{B(Q,r)\cap\partial\Omega}|(\nabla u_\varepsilon)^*|^p dS\bigg)^{1/p}
\leq C\bigg(\dashint_{B(Q,2r)\cap\partial\Omega}|(\nabla u_\varepsilon)^*|^2dS\bigg)^{1/2}
+  C\bigg(\dashint_{B(Q,2r)\cap\partial\Omega}
|(u_\varepsilon)^*|^2dS\bigg)^{1/2},
\end{equation}
whenever $u_\varepsilon\in H^1(B(Q,3r)\cap\Omega;\mathbb{R}^m)$ is a weak solution to
$\mathcal{L}_\varepsilon(u_\varepsilon) = 0$ in $B(Q,3r)\cap\Omega$ with
$\partial u_\varepsilon/\partial\nu_\varepsilon = 0$ on $B(Q,3r)\cap\partial\Omega$ for some
$Q\in\partial\Omega$ and $0<r<r_0$. Then the weak solutions to
$\mathcal{L}_\varepsilon(u_\varepsilon) = 0$ in $\Omega$ and
$\partial u_\varepsilon/\partial\nu_\varepsilon = g \in L^p(\partial\Omega;\mathbb{R}^m)$
satisfy the estimate
$\|(\nabla u_\varepsilon)^*\|_{L^p(\partial\Omega)}\leq C\|g\|_{L^p(\partial\Omega)}$.
\end{thm}

Compared to the homogeneous case, here we need to treat the quantity
``$\nabla u_\varepsilon+u_\varepsilon$'' as a whole.
The reason is that $u_\varepsilon$ as a solution is full certainty, and we can not
use Poincar\'e's inequality as freely as in the homogeneous case.
This point leads to the main technical difficulties in the paper.
In view of the above theorem, the problem is reduced to show the estimate
$\eqref{a:2.1}$, and it will be done by the following boundary estimate.

\begin{thm}[boundary Lipschitz estimates]\label{thm:1.0}
Let $\Omega$ be a bounded $C^{1,\eta}$ domain.
Suppose that the coefficients of $\mathcal{L}_\varepsilon$ satisfy the conditions
$\eqref{a:1}$, $\eqref{a:2}$, $\eqref{a:3}$ with $\lambda\geq\lambda_0$ and $A,V$ additionally satisfy $\eqref{a:4}$.
Let $u_\varepsilon\in H^1(B(Q,r)\cap\Omega;\mathbb{R}^m)$ be a weak solution
to $\mathcal{L}_\varepsilon(u_\varepsilon) = \emph{div}(f)+F$ in $B(Q,r)\cap\Omega$ with
$\partial u_\varepsilon/\partial\nu_\varepsilon = g-n\cdot f$ on $B(Q,r)\cap\partial\Omega$
for some $Q\in\partial\Omega$ and $0<r\leq 1$. Assume that
\begin{equation*}
\begin{aligned}
 \mathcal{R}(F,f,g;r)& =  r\Big(\dashint_{B(Q,r)\cap\Omega}|F|^p\Big)^{1/p}
 + \|f\|_{L^\infty(B(Q,r)\cap\partial\Omega)}
 + r^\sigma [f]_{C^{0,\sigma}(B(Q,r)\cap\partial\Omega)}\\
 &+ \|g\|_{L^\infty(B(Q,r)\cap\partial\Omega)}
 + r^\sigma [g]_{C^{0,\sigma}(B(Q,r)\cap\partial\Omega)} <\infty,
\end{aligned}
\end{equation*}
where $p>d$ and $0<\sigma\leq\eta<1$. Then we have
\begin{equation}\label{pri:1.0}
  \sup_{B(Q,r/2)\cap\Omega}|\nabla u_\varepsilon|
  \leq C\bigg\{\frac{1}{r}\Big(\dashint_{B(Q,r)\cap\Omega}|u_\varepsilon|^2\Big)^{1/2}
  + \mathcal{R}(F,f,g;r)\bigg\},
\end{equation}
where $C$ depends on $\mu,\kappa,\tau,\lambda,m,d$ and the character of $\Omega$.
\end{thm}

In fact, the first author has developed the global Lipschitz estimate
in \cite[Theorem 1.2]{X1}. The main idea is to construct the connection between the solutions
corresponding to $L_\varepsilon=\text{div}(A(x/\varepsilon)\nabla)$ and $\mathcal{L}_\varepsilon$ via
the Neumann boundary corrector (see \cite[pp.4371]{X1}), such that the regularity results on $L_\varepsilon$ can be applied to
$\mathcal{L}_\varepsilon$ directly. Thus, his proof of the global Lipschitz estimate avoids the
the stated estimate $\eqref{pri:1.0}$.

Generally speaking, if there are the global estimates in our hand,
the corresponding boundary estimates will be obtained simply by
using the localization technique as in \cite[Lemma 2.17]{X1}.
Unfortunately, the estimate $\eqref{pri:1.0}$ can not be easily achieved in this way,
even for homogeneous operator $L_\varepsilon$. Because,
\begin{equation*}
L_\varepsilon(w_\varepsilon) = - \text{div}\big[A(x/\varepsilon)\nabla\phi u_\varepsilon\big]
- A(x/\varepsilon)\nabla u_\varepsilon\nabla\phi   \quad\text{in~} \Omega,
\end{equation*}
where $w_\varepsilon = u_\varepsilon\phi$,
and $u_\varepsilon$ satisfies $L_{\varepsilon}(u_\varepsilon) = 0$ in $\Omega$ with
$\phi\in C_0^1(\mathbb{R}^d)$ being a cut-off function. It is clear to see that the first term in the
right-hand side involves ``$A(x/\varepsilon)$'', which will produce a factor $\varepsilon^{-\sigma}$
in a H\"older semi-norm with the index $\sigma\in(0,1)$.
Obviously, we need an additional effort to conceal this factor and we have no plan to show the related
techniques in this direction. Instead, we want to prove the estimate $\eqref{pri:1.0}$ based upon
a convergence rate coupled with the so-called Campanato iteration.
This method has been well studied in \cite{AM,ASC,ASZ,S5} for periodic and nonperiodic settings.
Compared to the compactness argument shown in \cite{MAFHL,MAFHL3}, we are released from estimating
the boundary correctors, which is usually a very tough work.

The main idea in the proof of Theorem $\ref{thm:1.0}$ is similar to that in \cite{AM,S5},
but the nonhomogeneous operator $\mathcal{L}_\varepsilon$ will cause
some critical differences and technical difficulties.
For example, the solution $u_\varepsilon$ to $(\mathbf{NH_\varepsilon})$ is assured by given data.
It made us employ the following quantity
\begin{equation*}
\inf_{M\in\mathbb{R}^{d\times d}}\dashint_{B(Q,r)\cap\Omega}|u_\varepsilon - Mx - \tilde{c}|^2
\quad \text{instead of} \quad
\inf_{M\in\mathbb{R}^{d\times d}\atop
c\in\mathbb{R}^d}\dashint_{B(Q,r)\cap\Omega}|u_\varepsilon - Mx - c|^2
\end{equation*}
to carry out the iteration program, where $Q\in\partial\Omega$ and $\varepsilon\leq r<1$.
Moreover, $\tilde{c}$ may be given by $u_0(Q)$, which is the approximating solution to
$\mathcal{L}_0(u_0) = \mathcal{L}_\varepsilon(u_\varepsilon)$ in $B(Q,r)\cap\Omega$ with $\partial u_0/\partial\nu_0
= \partial u_\varepsilon/\partial\nu_\varepsilon$ on $\partial (B(Q,r)\cap\Omega)$.
Thus saying the solution assured means that
\begin{equation*}
 |\tilde{c}| \leq C\Big(\dashint_{B(Q,r)\cap\Omega}|u_0|^2\Big)^{1/2}
 \leq C\Big(\dashint_{B(Q,2r)\cap\Omega}|u_\varepsilon|^2\Big)^{1/2} + \text{give~data},
\end{equation*}
where we also use the following approximating result (see Lemma $\ref{lemma:4.1}$)
\begin{equation*}
\begin{aligned}
 \Big(\dashint_{B(Q,r)\cap\Omega} |u_\varepsilon - u_0|^2  \Big)^{1/2}
 \leq C\left(\frac{\varepsilon}{r}\right)^{\rho}
 \bigg\{ \Big(\dashint_{B(Q,2r)\cap\Omega}|u_\varepsilon|^2 \Big)^{1/2}
 + \text{given~data}\bigg\}
\end{aligned}
\end{equation*}
with some $\rho\in(0,1)$. In order to continue the iteration,
let $v_\varepsilon = u_\varepsilon - \tilde{c} - \varepsilon\chi_0(x/\varepsilon)\tilde{c}$ and
$v_0 = u_0 - \tilde{c}$, and then we give a revised approximating lemma (see Lemma $\ref{lemma:4.4}$), which says
\begin{equation*}
\begin{aligned}
 \Big(\dashint_{B(Q,r)\cap\Omega} |v_\varepsilon - v_0|^2  \Big)^{1/2}
 \leq C\left(\frac{\varepsilon}{r}\right)^{\rho}
 \bigg\{
 \Big(\dashint_{B(Q,2r)\cap\Omega}|u_\varepsilon-\tilde{c}|^2 \Big)^{1/2}
  + r|\tilde{c}| + \text{give~data}\bigg\}.
\end{aligned}
\end{equation*}
Here we remark that if we regard the constant $\tilde{c}$ as the given data,
and it will play a role as $F$ and $g$ (see for example Remark $\ref{re:4.1}$).
Thus, it is equivalent to $|\nabla u_\varepsilon|$ or $|\nabla^2 u_\varepsilon|$
in the sense of rescaling, and that is the reason why we have a factor ``$r$''
in front of the constant $|\tilde{c}|$, and this factor is very important in the later iterations.
Also, we made a few modification on the iteration lemma (see Lemma $\ref{lemma:4.2}$),
which has been proved by Z. Shen in \cite{S5},
originally by S. Armstrong, C. Smart in \cite{ASC}. Then
a routine computation leads to a large scale estimate,
\begin{equation*}
\Big(\dashint_{B(Q,r)\cap\Omega}|\nabla u_\varepsilon|^2\Big)^{1/2}
\leq C\bigg\{\Big(\dashint_{B(Q,1)\cap\Omega}|u_\varepsilon|^2\Big)^{1/2}
+ \Big(\dashint_{B(Q,2r)\cap\Omega}|u_\varepsilon|^2\Big)^{1/2}+\text{given~data}\bigg\}
\end{equation*}
for any $\varepsilon\leq r<1$. Obviously, the second term in the right-hand side requires a uniform
control with respect to the scale $r$, and it would be done by a local $W^{1,p}$ estimate with $p>2$,
which involves the so-called bootstrap argument.
Consequently, the proof of $\eqref{pri:1.0}$ will be completed by a blow-up argument.
However, there is a gap between the desired estimate $\eqref{a:2.1}$
and the stated estimate $\eqref{pri:1.0}$, and our only recourse is the Neumann boundary corrector here.
We refer the reader to Lemma $\ref{lemma:6.1}$ for the details. Also, we mention that
if the symmetry condition $A=A^*$ is additionally assumed, then the Neumann boundary corrector will have
a better estimate (see Remark $\ref{re:6.1}$).
Up to now, we have specified the key points in the proof of Theorem $\ref{thm:1.1}$
for $p\geq 2$. We mention that the proof in the case $1<p<2$ can not been derived by duality arguments.
For given boundary atom data $g$ in $L^2$ Neumann problem, we need to establish the following estimate
\begin{equation*}
\int_{\partial\Omega} (\nabla u_\varepsilon)^* \leq C
\end{equation*}
(see Theorem $\ref{thm:6.1}$), which is based upon the decay estimates of Neumann functions.
Since we have investigated the fundamental solutions of $\mathcal{L}_\varepsilon$ in \cite{X2},
this part of the proof may follow from those in \cite{KFS1,DK} without any real difficulty.

In terms of Lipschitz domains, the well-posedness of $(\mathbf{NH_\varepsilon})$ may be known
whenever $p$ is closed to $2$. For a $C^1$ domain,
whether Theorem $\ref{thm:1.1}$ is correct or not is still an open question,
while it is true for homogenized system $(\mathbf{NH_0})$,
and the reader may find a clue in \cite[Section 3]{X2}.
We mention that $L^p$ Dirichlet problem on $\mathcal{L}_\varepsilon$ has already been given
by \cite[Theorem 1.4]{X0} in $C^{1,\eta}$ domains. The assumption of $d\geq 3$ is not essential but
convenient to organize the paper.
Finally, without attempting to be exhaustive, we refer the reader to
\cite{ABJLGP,CP,DK,GF,GNF,GX,HMT,VSO,KS2,aKS1,NSX,S4,S1,ST1,ZVVPSE} and references therein for more results.

This paper is organized as follows. Some definitions
and known lemmas and the proof of Theorem $\ref{thm:2.1}$ are
introduced in section 2. We show a convergence rate in section 3.
Section 4 is devoted to study boundary estimates and we prove some decay estimates of
Neumann functions in section 5. The proof of Theorem $\ref{thm:1.1}$ is
consequently given in the last section.

\section{Preliminaries}

Define the correctors $\chi_k = (\chi_{k}^{\alpha\beta})$ with $0\leq k\leq d$, related to $\mathcal{L}_\varepsilon$ as follows:
\begin{equation}
\left\{ \begin{aligned}
 &L_1(\chi_k) = \text{div}(V)  \quad \text{in}~ \mathbb{R}^d, \\
 &\chi_k\in H^1_{per}(Y;\mathbb{R}^{m^2})~~\text{and}~\int_Y\chi_k dy = 0
\end{aligned}
\right.
\end{equation}
for $k=0$, and
\begin{equation}
 \left\{ \begin{aligned}
  &L_1(\chi_k^\beta + P_k^\beta) = 0 \quad \text{in}~ \mathbb{R}^d, \\
  &\chi_k^\beta \in H^1_{per}(Y;\mathbb{R}^m)~~\text{and}~\int_Y\chi_k^\beta dy = 0
 \end{aligned}
 \right.
\end{equation}
for $1\leq k\leq d$, where $P_k^\beta = x_k(0,\cdots,1,\cdots,0)$ with 1 in the
$\beta^{\text{th}}$ position, $Y = (0,1]^d \cong \mathbb{R}^d/\mathbb{Z}^d$, and $H^1_{per}(Y;\mathbb{R}^m)$ denotes the closure
of $C^\infty_{per}(Y;\mathbb{R}^m)$ in $H^1(Y;\mathbb{R}^m)$.
Note that $C^\infty_{per}(Y;\mathbb{R}^m)$ is the subset of $C^\infty(Y;\mathbb{R}^m)$,
which collects all $Y$-periodic vector-valued functions. By asymptotic expansion arguments
(see \cite[pp.103]{ABJLGP} or \cite[pp.31]{VSO}), we obtain the homogenized operator
\begin{equation}\label{eq:2.1}
 \mathcal{L}_0 = -\text{div}(\widehat{A}\nabla+ \widehat{V}) + \widehat{B}\nabla + \widehat{c} + \lambda I,
\end{equation}
where $\widehat{A} = (\widehat{a}_{ij}^{\alpha\beta})$, $\widehat{V}=(\widehat{V}_i^{\alpha\beta})$,
$\widehat{B} = (\widehat{B}_i^{\alpha\beta})$ and $\widehat{c}= (\widehat{c}^{\alpha\beta})$ are given by
\begin{equation}\label{eq:2.2}
\begin{aligned}
\widehat{a}_{ij}^{\alpha\beta} = \int_Y \big[a_{ij}^{\alpha\beta} + a_{ik}^{\alpha\gamma}\frac{\partial\chi_j^{\gamma\beta}}{\partial y_k}\big] dy, \qquad
\widehat{V}_i^{\alpha\beta} = \int_Y \big[V_i^{\alpha\beta} + a_{ij}^{\alpha\gamma}\frac{\partial\chi_0^{\gamma\beta}}{\partial y_j}\big] dy, \\
\widehat{B}_i^{\alpha\beta} = \int_Y \big[B_i^{\alpha\beta} + B_j^{\alpha\gamma}\frac{\partial\chi_i^{\gamma\beta}}{\partial y_j}\big] dy, \qquad
\widehat{c}^{\alpha\beta} = \int_Y \big[c^{\alpha\beta} + B_i^{\alpha\gamma}\frac{\partial\chi_0^{\gamma\beta}}{\partial y_i}\big] dy.
\end{aligned}
\end{equation}

\begin{remark}
\emph{It is well known that $u_\varepsilon\to u_0$ strongly in $L^2(\Omega;\mathbb{R}^m)$,
where $u_0\in H^1(\Omega;\mathbb{R}^m)$ satisfies the equation
\begin{equation*}
(\mathbf{NH_0})\left\{
\begin{aligned}
\mathcal{L}_0(u_0) &= 0 &\quad &\text{in}~~\Omega, \\
 \frac{\partial u_0}{\partial\nu_0} &= g & \quad &\text{on} ~\partial\Omega,
\end{aligned}\right.
\end{equation*}
where $\partial/\partial\nu_0 = n\cdot\big(\widehat{A}\nabla +\widehat{V}\big)$,
(see for example \cite[pp.4374-4375]{X1}).}
\end{remark}

\begin{definition}\label{def:1}
\emph{The nontangential maximal function of $u$ is defined by
\begin{equation*}
(u)^*(Q) = \sup\big\{|u(x)|:x\in\Gamma_{N_0}(Q)\big\}
\qquad \forall Q\in\partial\Omega,
\end{equation*}
where $\Gamma_{N_0}(Q) = \big\{x\in\Omega:|x-Q|\leq N_0\delta(x)\big\}$ is the cone with vertex
$Q$ and aperture $N_0$, and $N_0>1$ is sufficiently large.}
\end{definition}

\begin{lemma}\label{lemma:2.3}
Suppose that the coefficients of $\mathcal{L}_\varepsilon$ satisfies $\eqref{a:1}$ and $\eqref{a:3}$
with $A\in \emph{VMO}(\mathbb{R}^d)$.
Let $u_\varepsilon$ be the solution of $\mathcal{L}_\varepsilon(u_\varepsilon) = 0$ in $\Omega$.
Then we have the following estimate
\begin{equation}\label{pri:2.1}
 (u_\varepsilon)^*(Q) \leq C\mathrm{M}_{\partial\Omega}(\mathcal{M}(u_\varepsilon))(Q)
\end{equation}
for any $Q\in\partial\Omega$,
where $C$ depends only on $\mu,\kappa,\lambda,m,d$ and $\|A\|_{\emph{VMO}}$.
\end{lemma}

\begin{remark}
\emph{The definition of $\text{VMO}(\mathbb{R}^d)$ may be found in \cite[pp.43]{S4}, and
the radial maximal function operator $\mathcal{M}$ is defined in \cite[Remark 2.21]{X1}.}
\end{remark}

\begin{proof}
Fixed $x\in\Lambda_{N_0}(Q)$,
the estimate $\eqref{pri:2.1}$ is based upon the interior estimate (see \cite[Corollary 3.5]{X0})
\begin{equation*}
\begin{aligned}
|u_\varepsilon(x)|
&\leq C\Big(\dashint_{B(x,r)}|u_\varepsilon|^2\Big)^{1/2} \\
& \leq C\dashint_{B(Q,c_0 r)\cap\partial\Omega}|\mathcal{M}(u_\varepsilon)|
\leq C\mathrm{M}_{\partial\Omega}(\mathcal{M}(u_\varepsilon))(Q),
\end{aligned}
\end{equation*}
where $r=\text{dist}(x,\partial\Omega)$, and $c_0>0$ is determined by $N_0$.
\end{proof}

\begin{lemma}\label{lemma:2.4}
Let $\Omega\subset\mathbb{R}^d$ be a bounded
Lipschitz domain, and $\mathcal{M}$ be defined
as the radical maximal function operator. Then
for any $h\in H^{1}(\Omega)$, we have the following estimate
\begin{equation}\label{pri:2.2}
\|\mathcal{M}(h)\|_{L^p(\partial\Omega)}
\leq C\|h\|_{W^{1,p}(\Omega)}
\end{equation}
where $C$ depends only on $d$ and
the character of $\Omega$.
\end{lemma}

\begin{proof}
It would be done by a few modification to the proof \cite[Lemma 2.24]{X1}.
\end{proof}

\begin{remark}
\emph{For the ease of the statement, we introduce the following notation.
\begin{equation*}
\begin{aligned}
D(Q,r) &= B(Q,r)\cap \Omega = \big\{(x^\prime,x_d)\in\mathbb{R}^d:|x^\prime|<r
~\text{and}~ \psi(x^\prime) < x_d < C_0r\big\},\\
\Delta(Q,r) &= B(Q,r)\cap \partial\Omega = \big\{(x^\prime,x_d)\in\mathbb{R}^d:|x^\prime|<r\big\},
\end{aligned}
\end{equation*}
where $\psi:\mathbb{R}^{d-1}\to\mathbb{R}$ is a Lipschitz or $C^{1,\eta}$ function.
We usually denote $D(Q,r)$ and $\Delta(Q,r)$ by $D_r$ and $\Delta_r$.}
\end{remark}

\begin{lemma}[A real method]\label{lemma:2.1}
Let $S_0$ be a cube of $\partial \Omega$ and $F\in L^2(2S_0)$. Let $p>2$ and $f\in L^q(2S_0)$ for some
$2<q<p$. Suppose that for each dyadic subcube $S$ of $S_0$ with $|S|$ with $|S|\leq \beta|S_0|$, there
exist two functions $F_S$ and $R_S$ on $2S$ such that $|F|\leq |F_S|+|R_S|$ on $2S$, and
\begin{equation}\label{pri:2.6}
\bigg\{\dashint_{2S}|R_S|^p\bigg\}^{1/p}
\leq C_1\bigg\{\Big(\dashint_{\alpha S}|F|^2\Big)^{1/2}
+\sup_{S^\prime\supset\supset S}\Big(\dashint_{S^\prime}|f|^2\Big)^{1/2}\bigg\},
\end{equation}
\begin{equation}\label{pri:2.7}
\dashint_{2S}|F_S|^2 \leq C_2 \sup_{S^\prime\subset S}\dashint_{S^\prime}|f|^2,
\end{equation}
where $C_1,C_2$ and $0<\beta<1<\alpha$. Then
\begin{equation}\label{pri:2.8}
\bigg\{\dashint_{S_0}|F|^q\bigg\}^{1/q}
\leq C\bigg\{\Big(\dashint_{2S_0}|F|^2\Big)^{1/2}+\Big(\dashint_{2S_0}|f|^q\Big)^{1/q}\bigg\},
\end{equation}
where $C>0$ depends only on $p,q,C_1,C_2,\alpha,\beta,d$ and the character of $\Omega$.
\end{lemma}

\begin{proof}
See for example \cite[Lemma 2.2]{aKS1}.
\end{proof}

\noindent\textbf{Proof of Theorem $\ref{thm:2.1}$}.
The main idea may be found in \cite[Lemma 9.2]{KFS2}, and
we make some modifications in the original proof to fit the case of nonhomogeneous operators.
To show the stated result, on account of a covering argument, it suffices to prove the following estimate
\begin{equation}\label{pri:2.5}
\bigg\{\dashint_{\Delta(Q,r)} |(\nabla u_\varepsilon)^*|^p dS\bigg\}^{1/p}
\leq C\bigg\{\dashint_{\Delta(Q,2r)}\Big(|(\nabla u_\varepsilon)^*|^2
+|(u_\varepsilon)^*|^2\Big)dS\bigg\}^{1/2}
+  C\bigg\{\dashint_{\Delta(Q,2r)}
|g|^pdS\bigg\}^{1/p}
\end{equation}
for any $0<r<r_0$, and it will accomplished by a real variable method originating in \cite{CP}
and further developed in \cite{S1,S2,S3}. Precisely speaking, we will apply Lemma $\ref{lemma:2.1}$ to our case.

Let $\chi_{\Delta_{8r}}$ represent the characteristic function of a set $\Delta_{8r}\subset\partial\Omega$,
where $r\in(0,r_0/100)$. Define $f = g\chi_{\Delta_{8r}}$, and
then we consider $u_\varepsilon = v_\varepsilon + w_\varepsilon$, in which
$v_\varepsilon$ and $w_\varepsilon$ satisfy $L^2$ Neumann problems
\begin{equation*}
(\text{I})\left\{
\begin{aligned}
\mathcal{L}_\varepsilon(v_\varepsilon) &=0 &\quad& \text{in}~\Omega,\\
\frac{\partial v_\varepsilon}{\partial\nu_\varepsilon} &= f  &\quad& \text{on}~\partial\Omega,
\end{aligned}
\right.\qquad
(\text{II})\left\{\begin{aligned}
\mathcal{L}_\varepsilon(w_\varepsilon) &=0 &\quad& \text{in}~\Omega,\\
\frac{\partial w_\varepsilon}{\partial\nu_\varepsilon} &= (1-\chi_{\Delta_{8r}})g
&\quad& \text{on}~\partial\Omega,
\end{aligned}
\right.
\end{equation*}
respectively.

For (\text{I}). It follows from the $L^2$ solvability (see \cite[Theorem 1.5]{X2}) that
\begin{equation*}
\dashint_{\Delta_{r}} |(\nabla v_\varepsilon)^*|^2
\leq \frac{C}{r^{d-1}}\int_{\partial\Omega} |(\nabla v_\varepsilon)^*|^2 dS
\leq \frac{C}{r^{d-1}}\int_{\partial\Omega} |f|^2 dS
\leq C\dashint_{\Delta_{8r}} |g|^2.
\end{equation*}
On the other hand, in view of the estimates $\eqref{pri:2.1}$ and $\eqref{pri:2.2}$, we have
\begin{equation*}
\dashint_{\Delta_{r}} |(v_\varepsilon)^*|^2
\leq C\dashint_{\Delta_{r}} |\mathcal{M}(v_\varepsilon)|^2
\leq \frac{C}{r^{d-1}}\int_{\Omega}\big(|\nabla v_\varepsilon|^2 + |v_\varepsilon|^2\big)dx
\leq  C\dashint_{\Delta_{r}} |g|^2.
\end{equation*}

Let $F_S = (\nabla v_\varepsilon)^* + (v)^*$, and combining the above two inequalities leads to
\begin{equation}\label{f:2.1}
 \dashint_{\Delta_{r}}|F_S|^2 \leq C \dashint_{\Delta_{8r}} |g|^2.
\end{equation}
This gives the estimate $\eqref{pri:2.7}$ in Lemma $\ref{lemma:2.1}$.

Observing (\text{II}), we have that $w_\varepsilon\in H^1(D_{3r};\mathbb{R}^m)$ satisfies
$\mathcal{L}_\varepsilon(w_\varepsilon) = 0$ in $D_{3r}$ with
$\partial w_\varepsilon/\partial\nu_\varepsilon = 0$ on $\Delta_{3r}$. Hence, it follows from the
reverse H\"older assumption $\eqref{a:2.1}$ that
\begin{equation}\label{f:2.2}
\begin{aligned}
\bigg(\dashint_{\Delta_r} |(\nabla w_\varepsilon)^*|^p \bigg)^{1/p}
&\leq C\bigg\{\dashint_{\Delta_{2r}} |(\nabla w_\varepsilon)^*
+(w_\varepsilon)^*|^2 \bigg\}^{1/2}\\
&\leq C\bigg\{\dashint_{\Delta_{2r}} |(\nabla u_\varepsilon)^*
+(u_\varepsilon)^*|^2 \bigg\}^{1/2} + C\bigg\{\dashint_{\Delta_{2r}} |F_S|^2\bigg\}^{1/2}\\
&\leq C\bigg\{\dashint_{\Delta_{2r}} |(\nabla u_\varepsilon)^*
+(u_\varepsilon)^*|^2 \bigg\}^{1/2} + C\bigg\{\dashint_{\Delta_{8r}} |g|^2\bigg\}^{1/2}.
\end{aligned}
\end{equation}
Meanwhile, by the boundary $L^\infty$ estimate $\eqref{pri:2.3}$ and \cite[Corollary 3.5]{X0}, one may have
\begin{equation}\label{f:2.3}
\begin{aligned}
\bigg\{\dashint_{\Delta_r}|(w_\varepsilon)^*|^p\bigg\}^{1/p}
&\leq C\bigg\{\dashint_{D_{2r}}|w_\varepsilon|^2\bigg\}^{1/2}
\leq C\bigg\{\dashint_{\Delta_{2r}}|(w_\varepsilon)^*|^2\bigg\}^{1/2} \\
&\leq C\bigg\{\dashint_{\Delta_{2r}}|(u_\varepsilon)^*|^2\bigg\}^{1/2}
 + C\bigg\{\dashint_{\Delta_{2r}}|(v_\varepsilon)^*|^2\bigg\}^{1/2} \\
&\leq C\bigg\{\dashint_{\Delta_{2r}} |(\nabla u_\varepsilon)^*
+(u_\varepsilon)^*|^2 \bigg\}^{1/2} + C\bigg\{\dashint_{\Delta_{8r}} |g|^2\bigg\}^{1/2},
\end{aligned}
\end{equation}
where we also use the estimate $\eqref{f:2.1}$ in the last inequality.

Let $R_S = (\nabla w_\varepsilon)^* + (w_\varepsilon)^*$, and it follows from the estimates
$\eqref{f:2.2}$ and $\eqref{f:2.3}$ that
\begin{equation}
\bigg\{\dashint_{\Delta_r}|R_S|^p\bigg\}^{1/p}
\leq C\bigg\{\dashint_{\Delta_{2r}} |F|^2 \bigg\}^{1/2} + C\bigg\{\dashint_{\Delta_{8r}} |g|^2\bigg\}^{1/2},
\end{equation}
where $F=(\nabla u_\varepsilon)^*+(u_\varepsilon)^*$, and this gives
the estimate $\eqref{pri:2.6}$. Thus, it is clear to see that
$F\leq F_S + R_S$ on $\partial\Omega$, and in terms of Lemma $\ref{lemma:2.1}$ we may have
\begin{equation}
\bigg\{\dashint_{\Delta_r}|F|^q\bigg\}^{1/q}
\leq C\bigg\{\dashint_{\Delta_{2r}} |F|^2 \bigg\}^{1/2} + C\bigg\{\dashint_{\Delta_{2r}} |g|^2\bigg\}^{1/2}.
\end{equation}
for any $2<q<p$, where we also employ a simple covering argument.
This implies the stated estimate $\eqref{pri:2.5}$, and we have completed the proof.
\qed

\begin{lemma}
Let $\Omega\subset\mathbb{R}^d$ be a bounded Lipschitz domain.
Suppose $A$ satisfies $\eqref{a:1}$. Let $u_\varepsilon\in H^1(\Omega;\mathbb{R}^m)$ be a weak solution
to $\mathcal{L}_\varepsilon(u_\varepsilon) = F$ in $\Omega$ and
$\partial u_\varepsilon /\partial\nu_\varepsilon = g$ on $\partial\Omega$, where
$F\in L^2(\Omega;\mathbb{R}^m)$ and $g\in L^2(\partial\Omega;\mathbb{R}^m)$. Then where exists
$p>2$ depending on $\mu,d$ and the character of $\Omega$, such that
\begin{equation}\label{pri:2.13}
\|\nabla u_\varepsilon\|_{L^p(\Omega)}
\leq C\Big\{\|F\|_{L^2(\Omega)}+\|g\|_{L^2(\partial\Omega)}\Big\},
\end{equation}
where $C$ depends on $\mu,\kappa,d,m$ and $\Omega$.
\end{lemma}

\begin{proof}
If $A\in\text{VMO}(\mathbb{R}^d)$ additionally satisfies $\eqref{a:2}$, one may show
\begin{equation}\label{f:2.7}
\|\nabla u_\varepsilon\|_{L^p(\Omega)}
\leq C\Big\{\|F\|_{L^q(\Omega)}+\|g\|_{B^{-1/p,p}(\partial\Omega)}\Big\}
\end{equation}
for $2\leq p<\infty$, with $1/q = 1/p+1/d$, which has been proved in \cite[Lemma 3.3]{X1}. Clearly,
we can choose $p>2$ close to $2$ such that $L^2(\Omega)\subset L^q(\Omega)$ and
$L^2(\partial\Omega)\subset B^{-1/p,p}(\partial\Omega)$, and this gives the estimate $\eqref{pri:2.13}$.
Note that without the periodicity and $\text{VMO}$ condition on $A$,
the estimate $\eqref{f:2.7}$ still holds for $|1/p-1/2|<\delta$,
where $\delta$ depends on $\mu,d$ and the character of $\Omega$, and we do not reproduce
the proof which is based upon a real method and reverse H\"older inequality (see \cite[Theorem 1.1.4]{S4}).
\end{proof}

\section{Convergence rates in Lipschitz domains}

\begin{thm}[convergence rates]\label{thm:3.1}
Let $\Omega\subset\mathbb{R}^d$ be a bounded Lipschitz domain. Suppose that
the coefficients satisfy $\eqref{a:1}$, $\eqref{a:2}$ and $\eqref{a:3}$.
Given $F\in L^{2}(\Omega;\mathbb{R}^m)$ and $g\in L^2(\partial\Omega;\mathbb{R}^m)$,
we assume that $u_\varepsilon,u_0\in H^1(\Omega;\mathbb{R}^m)$ satisfy
\begin{equation*}
(\emph{NH}_\varepsilon)\left\{\begin{aligned}
\mathcal{L}_\varepsilon(u_\varepsilon) & = F &\quad&\emph{in}~\Omega,\\
\frac{\partial u_\varepsilon}{\partial\nu_\varepsilon}
& = g &\quad & \emph{on}~\partial\Omega,
\end{aligned}\right.
\qquad
(\emph{NH}_0)\left\{\begin{aligned}
\mathcal{L}_0(u_0) & = F &\quad&\emph{in}~\Omega,\\
\frac{\partial u_0}{\partial\nu_0}
& = g &\quad & \emph{on}~\partial\Omega,
\end{aligned}\right.
\end{equation*}
respectively. Then we have
\begin{equation}\label{pri:3.1}
\|u_\varepsilon - u_0\|_{L^2(\Omega)}
\leq C\varepsilon^\rho
\Big\{\|F\|_{L^2(\Omega)}+\|g\|_{L^2(\partial\Omega)}\Big\},
\end{equation}
where $\rho>0$ and $C>0$ depend only on $\mu,\kappa,\lambda,m,d$ and $\Omega$.
\end{thm}

\begin{remark}
\emph{We mention that the results in this lemma do not depend on the symmetry condition $A=A^*$.
If it is assumed, then we have the convergence rate $O(\varepsilon\ln(r_0/\varepsilon))$
(see \cite[Theorem 1.2]{X3}). We introduce the following notation.
The co-layer set is $\Sigma_{r}=\{x\in\Omega:\text{dist}(x,\partial\Omega)>r\}$ and
$\Omega\setminus\Sigma_r$ is referred to as the layer part of $\Omega$.
We define the cut-off function $\psi_r\in C^1_0(\Omega)$ such that
$\psi_r = 1$ in $\Sigma_{2r}$,  $\psi_r = 0$ outside $\Sigma_{r}$ and
$|\nabla \psi_r|\leq C/r$.}
\end{remark}

\begin{lemma}\label{lemma:3.1}
Assume the same conditions as in Theorem $\ref{thm:3.1}$. Suppose that
the weak solutions $u_\varepsilon\in H^1(\Omega;\mathbb{R}^m)$ satisfies
$\mathcal{L}_\varepsilon(u_\varepsilon) = \mathcal{L}_0(u_0)$ in $\Omega$, and
$\partial u_\varepsilon/ \partial\nu_\varepsilon = \partial u_0/\partial\nu_0$ on $\partial\Omega$
with $u_0\in H^2(\Omega;\mathbb{R}^m)$. Let the first approximating corrector be defined by
\begin{equation}
w_\varepsilon = u_\varepsilon - u_0 - \varepsilon\chi_0(\cdot/\varepsilon)S_\varepsilon(\psi_{4\varepsilon} u_0)
-\varepsilon\chi_k(\cdot/\varepsilon)S_\varepsilon(\psi_{4\varepsilon}\nabla_k u_0),
\end{equation}
where $\psi_{4\varepsilon}$ is the cut-off function and
$S_\varepsilon$ is the smoothing operator (see \cite[Definition 2.10]{X1}).
Then we have
\begin{equation}\label{pri:3.4}
\|w_\varepsilon\|_{H^1(\Omega)}
\leq C\Big\{\|u_0\|_{H^1(\Omega\setminus\Sigma_{8\varepsilon})}
+\varepsilon\|u_0\|_{H^2(\Sigma_{4\varepsilon})}\Big\},
\end{equation}
where $C$ depends on $\mu,\kappa,\lambda,m,d$ and $\Omega$.
\end{lemma}

\begin{proof}
In fact, the desired result $\eqref{pri:3.4}$ has been shown in \cite[Lemma 5.3]{X1},
and its proof is too long to be reproduced here. We refer the reader to
\cite[Lemmas 5.2, 5.3]{X1} for the details.
\end{proof}

\begin{lemma}[layer $\&$ co-layer type estimates]\label{lemma:3.2}
Assume the same conditions as in Theorem $\ref{thm:3.1}$.
Let $u_0\in H^1(\Omega;\mathbb{R}^m)$ be the weak solution to $(\emph{NH}_0)$. Then
there exists $p>2$ such that
\begin{equation}\label{pri:3.2}
\|u_0\|_{H^1(\Omega\setminus\Sigma_{p_1\varepsilon})}
\leq C\varepsilon^{\frac{1}{2}-\frac{1}{p}}\Big\{\|F\|_{L^2(\Omega)}
+\|g\|_{L^2(\partial\Omega)}\Big\}
\end{equation}
and
\begin{equation}\label{pri:3.3}
\|\nabla^2 u_0\|_{L^2(\Sigma_{p_2\varepsilon})}
\leq C\varepsilon^{-\frac{1}{2}-\frac{1}{p}}\Big\{\|F\|_{L^2(\Omega)}
+\|g\|_{L^2(\partial\Omega)}\Big\}
\end{equation}
where $p_1,p_2>0$ are fixed real numbers, and $C$ depends on $\mu,d,p_1,p_2,\sigma,p$ and $\Omega$.
\end{lemma}

\begin{proof}
The main ideas may be found in \cite[Lemma 5.1.5]{S4},
and we provide a proof for the sake of the completeness. We first handle the layer type estimate
$\eqref{pri:3.2}$, and it follows from H\"older's inequality and the estimate $\eqref{pri:2.13}$ that
\begin{equation*}
\begin{aligned}
\|u_0\|_{H^1(\Omega\setminus\Sigma_{p_1\varepsilon})}
\leq C\varepsilon^{\frac{1}{2}-\frac{1}{p}}\|u_0\|_{W^{1,p}(\Omega)}
    \leq C\varepsilon^{\frac{1}{2}-\frac{1}{p}}\Big\{\|F\|_{L^2(\Omega)}+\|g\|_{L^2(\partial\Omega)}\Big\}.
\end{aligned}
\end{equation*}

On account of the interior estimate for $\mathcal{L}_0$, we have
\begin{equation}\label{f:3.20}
\dashint_{B(x,\delta(x)/4)}|\nabla^2 u_0|^2 dy
\leq \frac{C}{[\delta(x)]^{2}}\dashint_{B(x,\delta(x)/2)}|\nabla u_0|^2 dy +
C\dashint_{B(x,\delta(x)/2)}|u_0|^2 dy
+ C\dashint_{B(x,\delta(x)/2)}|F|^2 dy
\end{equation}
for any $x\in\Sigma_{p_2\varepsilon}$, where $\delta(x)=\text{dist}(x,\partial\Omega)$.
Since $|y-x|\leq \delta(x)/4$, it is not hard to see that $|\delta(x)-\delta(y)|\leq |x-y|\leq \delta(x)/4$ and
this implies $(4/5)\delta(y)< \delta(x)< (4/3)\delta(y)$. Therefore,
\begin{equation*}
\begin{aligned}
\int_{\Sigma_{p_2\varepsilon}}|\nabla^2 u_0|^2 dx
\leq \int_{\Sigma_{p_2\varepsilon}}\dashint_{B(x,\delta(x)/4)}|\nabla^2 u_0|^2dy dx
\leq \int_{\Sigma_{(p_2\varepsilon)/2}}|\nabla^2 u_0|^2 dx.
\end{aligned}
\end{equation*}
Then integrating both sides of $\eqref{f:3.20}$ over co-layer set $\Sigma_{p_2\varepsilon}$ leads to
\begin{equation*}
\begin{aligned}
\int_{\Sigma_{p_2\varepsilon}}|\nabla^2 u_0|^2 dx
&\leq C\int_{\Sigma_{(p_2\varepsilon)/2}}|\nabla u_0|^2 [\delta(x)]^{-2} dx
+ C\int_\Omega |F|^2 dx + C\int_\Omega |u_0|^2 dx \\
&\leq C\varepsilon^{-1-\frac{2}{p}}\Big(\int_{\Omega}
|\nabla u_0|^pdx\Big)^{\frac{2}{p}}
+ C\int_\Omega |F|^2 dx + C\int_\Omega |u_0|^2 dx,
\end{aligned}
\end{equation*}
and this together with $\eqref{pri:2.13}$ and $H^1$ estimate (see \cite[Lemma 3.1]{X1}) gives the
stated estimate $\eqref{pri:3.3}$. We have completed the proof.
\end{proof}

\noindent\textbf{Proof of Theorem $\ref{thm:3.1}$.}
On account of Lemmas $\ref{lemma:3.1},\ref{lemma:3.2}$, it is not hard to see that
\begin{equation}\label{f:3.1}
\begin{aligned}
\|u_\varepsilon - u_0\|_{L^2(\Omega)}
&\leq C\Big\{\|u_0\|_{H^1(\Omega\setminus\Sigma_{8\varepsilon})}
+ \varepsilon\|\nabla^2 u_0\|_{L^2(\Sigma_{4\varepsilon})} + \varepsilon\|u_0\|_{H^1(\Omega)}\Big\}\\
&\leq C\varepsilon^{\frac{1}{2}-\frac{1}{p}}\Big\{\|F\|_{L^2(\Omega)}+\|g\|_{L^2(\partial\Omega)}
\Big\},
\end{aligned}
\end{equation}
where we employ the estimate $H^1$ estimate. Let $\rho = 1/2-1/p$, and we have completed the proof.
\qed

\begin{corollary}\label{cor:3.1}
Assume the same conditions as in Theorem $\ref{thm:3.1}$.
For any $\xi\in\mathbb{R}^m$,
let $v_\varepsilon = u_\varepsilon - \xi - \varepsilon\chi_0(x/\varepsilon)\xi$ and
$v_0 = u_0 -\xi$, where $u_\varepsilon$ and $u_0$ satisfy
$(\emph{NH}_\varepsilon)$ and $(\emph{NH}_0)$, respectively. Then we have
\begin{equation}
\|v_\varepsilon - v_0\|_{L^2(\Omega)}
\leq C\varepsilon^\rho
\Big\{\|F\|_{L^2(\Omega)}+\|g\|_{L^2(\partial\Omega)}+|\xi|\Big\},
\end{equation}
where $\rho>0$ and $C>0$ depend only on $\mu,\kappa,\lambda,m,d$ and $\Omega$.
\end{corollary}

\begin{remark}
\emph{Let $v_\varepsilon$ and $v_0$ be given in Corollary $\ref{cor:3.1}$.
Then one may have the following equations}
\begin{equation}\label{eq:3.1}
\left\{\begin{aligned}
\mathcal{L}_\varepsilon (v_\varepsilon) & = F
+ \varepsilon\emph{div}\big(V_\varepsilon\chi_{0,\varepsilon}\xi\big)
- \big[B_\varepsilon(\nabla\chi_0)_\varepsilon  + c_\varepsilon+\lambda I\big]\xi
-\varepsilon\chi_{0,\varepsilon}\big[c_\varepsilon+\lambda I\big]\xi
\quad &\emph{in}&~\Omega, \\
\frac{\partial v_\varepsilon}{\partial \nu_\varepsilon}
& = g -\varepsilon n\cdot V_\varepsilon\chi_{0,\varepsilon}\xi
-n\cdot\big[A_\varepsilon(\nabla\chi_0)_\varepsilon+V_\varepsilon\big]\xi
\quad &\emph{on}&~\partial\Omega,
\end{aligned}\right.
\end{equation}
\emph{and}
\begin{equation}\label{eq:3.2}
\left\{\begin{aligned}
\mathcal{L}_0 (v_0) & = F - (\widehat{c} + \lambda I)\xi
\quad &\emph{in}&~\Omega, \\
\frac{\partial v_0}{\partial \nu_0}
& = g
-n\cdot\widehat{V}\xi
\quad &\emph{on}&~\partial\Omega,
\end{aligned}\right.
\end{equation}
\emph{
in which such the notation $V_\varepsilon = V(x/\varepsilon)$
and $\chi_{0,\varepsilon} = \chi_0(x/\varepsilon)$
follow the same simplified way as in \cite[Remark 2.15]{X1}.
Here we plan to give some simple computations as a preparation. Recalling the form of $\widehat{c}$ in
$\eqref{eq:2.2}$, let $\Delta \vartheta_0 = \widehat{c} - c - B\nabla\chi_0$ in $Y$ and
$\int_Y\vartheta_0(y)dy = 0$. This implies $\vartheta_0\in H^2_{loc}(\mathbb{R}^d)$
(see \cite[Remark 2.7]{X1}). Also, set $b_{0} = \widehat{V} - V(y) - A(y)\nabla\chi_0(y)$, and
$n\cdot b_0(y) = \frac{\varepsilon}{2}\big[n_i\frac{\partial}{\partial x_j} - n_j\frac{\partial}{\partial x_i}
\big]E_{ji0}(y)$, where $E_{ji0}$ is referred to as the dual correctors and $y=x/\varepsilon$
(see \cite[Lemma 4.4]{X1}).
Hence, there hold
\begin{equation}\label{eq:3.3}
\left\{\begin{aligned}
\mathcal{L}_0(v_0)
&= \mathcal{L}_\varepsilon(v_\varepsilon) - \varepsilon\text{div}\big[V(y)\chi_{0}(y)
+ (\nabla \vartheta_0)(y)\big]\xi
+\varepsilon\chi_{0}(y)\big[ c(y)+\lambda I\big]\xi
\quad &\text{in}&~~\Omega, \\
\frac{\partial v_0}{\partial\nu_0}
& = \frac{\partial v_\varepsilon}{\partial\nu_\varepsilon} + \varepsilon n\cdot V(y)\chi_{0}(y)\xi
- \frac{\varepsilon}{2}\Big[n_i\frac{\partial}{\partial x_j}
-n_j\frac{\partial}{\partial x_i}\Big]E_{ji0}(y) \xi
\quad &\text{on}&~\partial\Omega,
\end{aligned}\right.
\end{equation}}
\emph{and it will benefit the later discussion in the approximating lemma.}
\end{remark}

\section{Local boundary estimates}

\begin{thm}[Lipschitz estimates at large scales]\label{thm:4.1}
Let $\Omega$ be a bounded $C^{1,\eta}$ domain.
Suppose that the coefficients of $\mathcal{L}_\varepsilon$ satisfy
$\eqref{a:1}$, $\eqref{a:2}$, and $\eqref{a:3}$.
Let $u_\varepsilon\in H^1(D_5;\mathbb{R}^m)$ be a weak solution of
$\mathcal{L}_\varepsilon(u_\varepsilon) = F$ in $D_5$ and
$\partial u_\varepsilon/\partial\nu_\varepsilon = g$ on $\Delta_5$,
where $F\in L^p(D_5;\mathbb{R}^m)$ with $p>d$,
and $g\in C^{0,\sigma}(\Delta_5;\mathbb{R}^m)$ with $0<\sigma\leq\eta<1$.
Then there holds
\begin{equation}\label{pri:4.6}
\begin{aligned}
\Big(\dashint_{D_r} |\nabla u_\varepsilon|^2dx\Big)^{\frac{1}{2}}
\leq C\bigg\{\Big(\dashint_{D_1} |u_\varepsilon|^2dx\Big)^{\frac{1}{2}}
+ \Big(\dashint_{D_{2r}} |u_\varepsilon|^2dx\Big)^{\frac{1}{2}}
+ \Big(\dashint_{D_1} |F|^pdx\Big)^{\frac{1}{p}}
+ \|g\|_{C^{0,\sigma}(\Delta_1)}\bigg\}
\end{aligned}
\end{equation}
for any $\varepsilon\leq r<(1/4)$,
where $C$ depends only on $\mu, \lambda, \kappa, d, m, p$ and the character of $\Omega$.
\end{thm}

\begin{lemma}[boundary Caccioppoli's inequality]\label{lemma:4.3}
Let $\Omega\subset\mathbb{R}^d$ be a bounded Lipschitz domain. Suppose that
the coefficients of $\mathcal{L}_\varepsilon$ satisfy $\eqref{a:1}$ and $\eqref{a:3}$ with
$\lambda\geq\lambda_0$. Let $u_\varepsilon\in H^1(D_2;\mathbb{R}^m)$ be a weak
solution of $\mathcal{L}_\varepsilon(u_\varepsilon) = \emph{div}(f)+F$ in $D_2$ with
$\partial u_\varepsilon/\partial\nu_\varepsilon = g-n\cdot f$ on $\Delta_2$. Then there holds
\begin{equation}\label{pri:2.12}
\Big(\dashint_{D_r}|\nabla u_\varepsilon|^2\Big)^{1/2}
\leq C_{\mu}\bigg\{\frac{1}{r}\Big(\dashint_{D_{2r}}|u_\varepsilon|^2\Big)^{1/2}
+ \Big(\dashint_{D_{2r}}|f|^2\Big)^{1/2}
+ r\Big(\dashint_{D_{2r}}|F|^2\Big)^{1/2}
+ \Big(\dashint_{\Delta_{2r}}|g|^2\Big)^{1/2}\bigg\}
\end{equation}
for any $0<r\leq 1$, where $C_\mu$ depends only on $\mu,d,m$, and the character of $\Omega$.
\end{lemma}

\begin{remark}
\emph{The condition $\lambda\geq\lambda_0$ guarantees that the constant $C_\mu$ in $\eqref{pri:2.12}$
do not depend on $\kappa$, which may lead to a scaling-invariant estimate even for the case $r>1$
(see \cite[Lemma 2.7]{X2}). However, we do not seek
such the convenience here. Also, we mention that the range of $0<r\leq 1$ is necessary in our proof.}
\end{remark}

\begin{proof}
By rescaling arguments we may prove the result for $r=1$.
The proof is quite similar to that given for \cite[Lemma 2.7]{X0}, and it is not hard to
derive that
\begin{equation*}
\begin{aligned}
\frac{\mu}{2}\int_{D_2} & \phi^2|\nabla u_\varepsilon|^2 dx
+  (\lambda - \lambda_0)\int_{D_2} \phi^2 |u_\varepsilon|^2 dx \\
&\leq  C_\mu\int_{D_2} |\nabla\phi|^2|u_\varepsilon|^2 dx
+C_{\mu}\int_{D_2} \phi^2|f|^2 dx + \int_{D_2} \phi^2|F||u_\varepsilon| dx
+ \underbrace{\int_{\Delta_2} \phi^2g u_\varepsilon dS}_{I},
\end{aligned}
\end{equation*}
where $\phi\in C_0^1(\mathbb{R}^d)$ is a cut-off function satisfying
$\phi = 1$ in $D_1$ and $\phi = 0$ outside $D_{3/2}$ with $|\nabla\phi|\leq C$, and
the ellipticity condition $\eqref{a:1}$ coupled with integration by parts has been used in the computations.
Note that the last term $I$ is the new thing compared to the proof in \cite[Lemma 2.7]{X0}, and
the reminder of the proof is standard. Thus, we have that
\begin{equation*}
I \leq \frac{\mu}{10}\int_{D_2}|\phi\nabla u_\varepsilon|^2 dx
+ C\int_{D_2}|\phi u_\varepsilon|^2 dx + C\int_{\Delta_2} |\phi g|^2 dS
\end{equation*}
Note that the constant $C$ actually depends on $\mu,m,d$ and the character of $\Omega$.
Thus we can not use $\lambda_0$ to absorb this constant, which also means we can not deal with
the case $r>1$ by simply using the rescaling argument. We have completed the proof.
\end{proof}

\begin{remark}\label{re:4.1}
\emph{Assume the same conditions and $u_\varepsilon$ as in Lemma $\ref{lemma:4.3}$. Let
$v_\varepsilon = u_\varepsilon - \xi - \varepsilon\chi_0(x/\varepsilon)\xi$ satisfy
$\eqref{eq:3.1}$ in $D_2$. Then there holds
\begin{equation}\label{pri:4.7}
\Big(\dashint_{D_r}|\nabla v_\varepsilon|^2\Big)^{1/2}
\leq C_{\mu}\bigg\{\frac{1}{r}\Big(\dashint_{D_{2r}}|v_\varepsilon|^2\Big)^{1/2}
+ \Big(\dashint_{D_{2r}}|f|^2\Big)^{1/2}
+ r\Big(\dashint_{D_{2r}}|F|^2\Big)^{1/2}
+ \Big(\dashint_{\Delta_{2r}}|g|^2\Big)^{1/2} + |\xi|\bigg\}
\end{equation}
for any $0<r\leq 1$, where $C_\mu$ depends only on $\mu,d,m$, and the character of $\Omega$.}
\end{remark}

\begin{lemma}[local $W^{1,p}$ boundary estimate]\label{lemma:2.2}
Let $\Omega\subset\mathbb{R}^d$ be a bounded $C^1$ domain, and $2< p<\infty$.
Suppose that the coefficients of $\mathcal{L}_\varepsilon$ satisfy $\eqref{a:1}$ and $\eqref{a:3}$, and
$A\in\emph{VMO}(\mathbb{R}^d)$ additionally satisfies $\eqref{a:2}$.
Given $f\in L^p(D_2;\mathbb{R}^{md})$, $F\in L^q(D_2;\mathbb{R}^m)$ with $q=\frac{pd}{d+p}$ and
$g\in L^\infty(\Delta_2;\mathbb{R}^m)$, define a local source quantity as
\begin{equation*}
\mathcal{R}_p(f,F,g;r) = \Big(\dashint_{D_r}|f|^p\Big)^{1/p} + r\Big(\dashint_{D_r}|F|^q\Big)^{1/q}
+ \|g\|_{L^\infty(\Delta_r)}
\end{equation*}
for any $0<r\leq 1$.
Let $u_\varepsilon\in H^1(D_2;\mathbb{R}^m)$ be the weak solution to
$\mathcal{L}_\varepsilon(u_\varepsilon) = \emph{div}(f)+ F$ in $D_2$ and
$\partial u_\varepsilon/\partial\nu_\varepsilon = g -n\cdot f$ on $\Delta_2$ with the local boundedness
assumption
\begin{equation}\label{a:2.2}
  \|u_\varepsilon\|_{W^{1,2}(D_1)}+ \mathcal{R}_p(f,F,g;1) \leq  1.
\end{equation}
Then, there exists $C_p>0$, depending on $\mu,\kappa,\lambda,m,d,p,\|A\|_{\emph{VMO}}$ and
the character of $\Omega$, such that
\begin{equation}\label{pri:2.10}
\|u_\varepsilon\|_{W^{1,p}(D_{1/2})} \leq C_p.
\end{equation}
\end{lemma}

\begin{proof}
The proof is based upon the localization technique coupled with a bootstrap argument which may be found
in \cite[Lemma 2.19]{X1} and \cite[Theorem 3.3]{X0}. Let $w_\varepsilon =\phi u_\varepsilon$, where
$\phi\in C_0^1(\mathbb{R}^d)$ be a cut-off function satisfying $\phi = 1$ in $D_{1/2}$ and
$\phi = 0$ outside $D_{1}$ with $|\nabla\phi|\leq C$. Then we have
\begin{equation*}
\left\{\begin{aligned}
\mathcal{L}_\varepsilon(w_\varepsilon) &= \text{div}(\tilde{f}) + \tilde{F} &\quad&\text{in} ~~D_2,\\
\frac{\partial w_\varepsilon}{\partial\nu_\varepsilon}
&= \big(\frac{\partial u_\varepsilon}{\partial\nu_\varepsilon}\big)\phi - n\cdot \tilde{f}
&\quad&\text{on}~\partial D_2,
\end{aligned}\right.
\end{equation*}
where
\begin{equation*}
 \tilde{f} = f\phi - A(x/\varepsilon)\nabla\phi u_\varepsilon,\quad
 \tilde{F} = F\phi - f\cdot\nabla\phi -A(x/\varepsilon)\nabla u_\varepsilon\nabla\phi
 +\big[B(x/\varepsilon)-V(x/\varepsilon)\big]\nabla\phi u_\varepsilon.
\end{equation*}
Thus, according to the global $W^{1,p}$ estimate (see \cite[Theorem 3.1]{X1}), we may obtain
\begin{equation}\label{f:2.5}
\|u_\varepsilon\|_{W^{1,p}(D_{1/2})}\leq \|w_\varepsilon\|_{W^{1,p}(D_{1/2})}
\leq C\Big\{\|u_\varepsilon\|_{W^{1,q}(D_1)}+\mathcal{R}_p(f,F,g;1)\Big\}
\end{equation}
where we use the Sobolev embedding theorem
$\|u_\varepsilon\|_{L^p(D_{1})}\leq C\|u_\varepsilon\|_{W^{1,q}(D_1)}$ with $q=\frac{pd}{p+d}$.

The interval $[1/2,1]$ may be divided into $1/2\leq r_1<\cdots <r_i<r_{i+1}<\cdots<r_{k_0}\leq 1$
with $i=1,\cdots,k_0$, where $k_0=\big[\frac{d}{2}\big]+1$ denotes the times of iteration, and
$[\frac{d}{2}]$ represents the integer part of $d/2$. By choosing the cut-off function
$\phi_i\in C_0^1(\mathbb{R}^d)$ such that $\phi_i =1$ in $D_{r_i}$ and
$\phi_i =0$ outside $D_{r_{i+1}}$ with $|\nabla\phi_i|\leq C/(r_{i+1}-r_i)$, one may derive
that
\begin{equation}\label{f:2.6}
\begin{aligned}
\|u_\varepsilon\|_{W^{1,p}(D_{1/2})}
\leq \cdots
&\leq \| w_\varepsilon\|_{W^{1,p_i}(D_{r_i})} + \cdots \\
&\leq C_{p_i}\Big\{\|u_\varepsilon\|_{W^{1,p_{i-1}}(D_{r_{i+1}})}
+\mathcal{R}_{p_i}(f,F,g;r_{i+1})\Big\}
+ C(d)\mathcal{R}_p(f,F,g;1) \\
\leq \cdots
&\leq
C\Big\{\|u_\varepsilon\|_{W^{1,2}(D_{1})}+\mathcal{R}_p(f,F,g;1)\Big\}\leq C
\end{aligned}
\end{equation}
where $p_i = 2d/(d-2i)$ and we note that there are two cases $p>p_{k_0} = \frac{2d}{d-2k_0}$
and $p\in(2,p_{k_0}]$ should be discussed. We refer the reader to \cite[Theorem 3.3]{X0} for the details.
Also, to obtain the second line of  $\eqref{f:2.6}$ we use the following fact that
\begin{equation*}
\mathcal{R}_{p_i}(f,F,g;r_i) \leq  C(d)\mathcal{R}_{p}(f,F,g;1)
\end{equation*}
for any $2<p_i\leq p$ and $r_i\in[1/2,1]$. We end the proof here.
\end{proof}

\begin{corollary}
Assume the same conditions as in Lemma $\ref{lemma:2.2}$. Let $0<\sigma<1$, and $p=d/(1-\sigma)$.
Suppose that $u_\varepsilon\in H^1(D_2;\mathbb{R}^m)$ is a weak solution of
$\mathcal{L}_\varepsilon(u_\varepsilon) = \emph{div}(f)+F$ in $D_2$ and
$\partial u_\varepsilon/\partial\nu_\varepsilon = g-n\cdot f$ on $\Delta_2$ with the local boundedness
assumption $\eqref{a:2.2}$. Then we have the boundary H\"older  estimate
\begin{equation}\label{pri:2.11}
 \|u_\varepsilon\|_{C^{0,\sigma}(D_{1/2})} \leq C_\sigma,
\end{equation}
where $C_\sigma$ depends on $\mu,\kappa,\lambda,m,d,\sigma,\|A\|_{\emph{VMO}}$
and the character of $\Omega$. In particularly, for any $s>0$ there holds
\begin{equation}\label{pri:2.3}
  \|u_\varepsilon\|_{L^\infty(D_{r/2})}\leq C\bigg\{\Big(\dashint_{D_{r}}|u_\varepsilon|^s\Big)^{1/s}
  + r\mathcal{R}_p(f,F,g;r)\bigg\}
\end{equation}
for any $0<r\leq 1$, where $C$ depends on $s$ and
$C_\sigma$.
\end{corollary}

\begin{proof}
The estimate $\eqref{pri:2.11}$ directly follows from the Sobolev embedding theorem and the estimate
$\eqref{pri:2.10}$. To show the estimate $\eqref{pri:2.3}$ we also employ Caccioppoli's inequality
$\eqref{pri:2.12}$ and a rescaling argument. The details may be found in \cite[Corollary 3.5]{X0} and we
do not reproduce here.
\end{proof}

\begin{lemma}[approximating lemma]\label{lemma:4.1}
Let $\varepsilon\leq r<1$.
Assume the same conditions as in Theorem $\ref{thm:4.1}$.
Let $u_\varepsilon\in H^1(D_{2r};\mathbb{R}^m)$ be a weak solution of
$\mathcal{L}_\varepsilon(u_\varepsilon) = F$ in $D_{2r}$
and $\partial u_\varepsilon/\partial\nu_\varepsilon = g$ on $\Delta_{2r}$.  Then there exists
$w\in H^1(D_{r};\mathbb{R}^m)$ such that
$\mathcal{L}_0(w) = F$ and
$\partial w/\partial\nu_0 = g$ on $\Delta_{r}$,
and there holds
\begin{equation}\label{pri:4.1}
\begin{aligned}
 \Big(\dashint_{D_r} |u_\varepsilon - w|^2  \Big)^{1/2}
 \leq C\left(\frac{\varepsilon}{r}\right)^{\rho}
 \bigg\{
 &\Big(\dashint_{D_{2r}}|u_\varepsilon|^2 \Big)^{1/2}
 + r^2\Big(\dashint_{D_{2r}} |F|^2\Big)^{1/2}
 + r\Big(\dashint_{\Delta_{2r}}|g|^2\Big)^{1/2}\bigg\},
\end{aligned}
\end{equation}
where $\rho\in(0,1/2)$
and $C>0$ $\mu,\lambda,\kappa,d,m$ and the the character of $\Omega$.
\end{lemma}

\begin{proof}
The idea may be found in \cite[Theorem 5.1.1]{S4}.
By rescaling argument one may assume $r=1$. For any $t\in (1,3/2)$,
there exists $w\in H^1(D_t;\mathbb{R}^m)$ satisfying
$\mathcal{L}_0(w)= F$ in $D_t$,
and $\partial w/\partial\nu_0 = \partial u_\varepsilon/\partial\nu_\varepsilon$ on $\partial D_t$.
In view of Theorem $\ref{thm:3.1}$, we have
\begin{equation}\label{f:4.5}
\begin{aligned}
\|u_\varepsilon - w\|_{L^2(D_t)}
\leq C\varepsilon^{\rho}\bigg\{
 \|F\|_{L^{2}(D_t)}
+\|g\|_{L^{2}(\Delta_2)}
+\|u_\varepsilon\|_{W^{1,2}(\partial D_t\setminus\Delta_2)}\bigg\},
\end{aligned}
\end{equation}
and it remains to estimate the last term in the right-hand side of $\eqref{f:4.5}$. Due to
the estimate $\eqref{pri:2.12}$ and co-area formula, we have
\begin{equation}\label{f:4.1}
\|u_\varepsilon\|_{W^{1,2}(\partial D_t\setminus\Delta_2)} \leq C\Big\{\|u_\varepsilon\|_{L^2(D_2)}
+ \|F\|_{L^2(D_2)} + \|g\|_{L^{2}(\Delta_2)}\Big\}
\end{equation}
for some $t\in(1,3/2)$.
Hence, combining $\eqref{f:4.5}$ and $\eqref{f:4.1}$ we acquire
\begin{equation*}
\begin{aligned}
\big\|u_\varepsilon - w\big\|_{L^2(D_1)}
&\leq C\varepsilon^{\rho}
\bigg\{\Big(\dashint_{D_{2}}|u_\varepsilon|^2 \Big)^{1/2}
+ \Big(\dashint_{D_{2}}|F|^2 \Big)^{1/2}
+ \Big(\dashint_{\Delta_{2}}|g|^2\Big)^{1/2} \bigg\}.
\end{aligned}
\end{equation*}
By rescaling argument we can derive
the desired estimate $\eqref{pri:4.1}$, and we complete the proof.
\end{proof}

\begin{lemma}[revised approximating lemma]\label{lemma:4.4}
Let $\varepsilon\leq r<1$.
Assume the same conditions as in Theorem $\ref{thm:4.1}$.
Let $u_\varepsilon\in H^1(D_{2r};\mathbb{R}^m)$ be a weak solution of
$\mathcal{L}_\varepsilon(u_\varepsilon) = F$ in $D_{2r}$
and $\partial u_\varepsilon/\partial\nu_\varepsilon = g$ on $\Delta_{2r}$.
Let $v_\varepsilon = u_\varepsilon - \xi - \varepsilon\chi_0(x/\varepsilon)\xi$
for some $\xi\in\mathbb{R}^d$.
Then there exists
$v_0 = u_0-\xi\in H^1(D_{r};\mathbb{R}^m)$ such that
the equation $\eqref{eq:3.3}$ holds in $D_r$, and we have
\begin{equation}\label{pri:4.8}
\begin{aligned}
 \Big(\dashint_{D_r} |v_\varepsilon - v_0|^2  \Big)^{1/2}
 \leq C\left(\frac{\varepsilon}{r}\right)^{\rho}
 \bigg\{
 &\Big(\dashint_{D_{2r}}|u_\varepsilon-\xi|^2 \Big)^{1/2}
 + r^2\Big(\dashint_{D_{2r}} |F|^2\Big)^{1/2}
 + r\Big(\dashint_{\Delta_{2r}}|g|^2\Big)^{1/2} + r|\xi|\bigg\},
\end{aligned}
\end{equation}
where $\rho\in(0,1/2)$
and $C>0$ depend only on $\mu,\lambda,\kappa,d,m$ and the the character of $\Omega$.
\end{lemma}

\begin{proof}
Here we need to employ Caccioppoli's inequality $\eqref{pri:4.7}$ for $v_\varepsilon$, and Corollary $\ref{cor:3.1}$.
The rest of the proof is as the same as the previous lemma, and we omit the proof.
\end{proof}

Before we proceed further, for any matrix $M\in \mathbb{R}^{m\times d}$,
we denote $G(r,v)$ as the following
\begin{equation}
\begin{aligned}
G(r,v) &= \frac{1}{r}\inf_{M\in\mathbb{R}^{d\times d}}
\Bigg\{\Big(\dashint_{D_r}|v-Mx-\tilde{c}|^2dx\Big)^{\frac{1}{2}}
+ r^2\Big(\dashint_{D_r}|F|^p\Big)^{\frac{1}{p}}
+ r^2\Big(\dashint_{D_r}|Mx+\tilde{c}|^p\Big)^{\frac{1}{p}} \\
&\qquad + r^2 |M|  + r\Big\|g-\frac{\partial}{\partial\nu_0}\big(Mx+\tilde{c}\big)\Big\|_{L^\infty(\Delta_{r})}
+r^{1+\sigma}\Big[g-\frac{\partial}{\partial\nu_0}\big(Mx+\tilde{c}\big)\Big]_{C^{0,\sigma}(\Delta_r)}\Bigg\},
\end{aligned}
\end{equation}
where we set $\tilde{c} = u_0(0)$.

\begin{lemma}\label{lemma:5.5}
Let $u_0\in H^1(D_2;\mathbb{R}^m)$ be a solution of
$\mathcal{L}_0(u_0) = F$ in $D_2$ and
$\partial u_0/\partial\nu_0 = g$ on $\Delta_2$, where $g\in C^{0,\sigma}(\Delta_2;\mathbb{R}^m)$. Then
there exists $\theta\in(0,1/4)$, depending on $\mu,d,\kappa,\lambda,m,d$ and the character of $\Omega$,
such that
\begin{equation}\label{pri:5.10}
 G(\theta r,u_0)\leq \frac{1}{2} G(r,u_0)
\end{equation}
holds for any $r\in(0,1)$.
\end{lemma}

\begin{proof}
We may assume $r=1$ by rescaling argument. By the definition of $G(\theta,u_0)$, we see that
\begin{equation*}
\begin{aligned}
G(\theta,u_0) &\leq  \frac{1}{\theta}
\bigg\{\Big(\dashint_{D_\theta}|u_0-M_0x-\tilde{c}|^2\Big)^{\frac{1}{2}}
+ \theta^2\Big(\dashint_{D_\theta}|F|^p\Big)^{\frac{1}{p}}
+ \theta^2\Big(\dashint_{D_\theta}|M_0 x+\tilde{c}|^p\Big)^{\frac{1}{p}}+\theta^2 |M_0|\\
&\quad
+ \theta\Big\|\frac{\partial}{\partial\nu_0}\big(u_0-M_0x-\tilde{c}\big)\Big\|_{L^\infty(\Delta_{\theta})}
+\theta^{1+\sigma}\Big[\frac{\partial}{\partial\nu_0}\big(u_0-M_0x-\tilde{c}\big)\Big]_{C^{0,\sigma}(\Delta_\theta)}\bigg\} \\
&\leq \theta^{\sigma}\bigg\{\|u_0\|_{C^{1,\sigma}(D_{1/2})}
+\Big(\dashint_{D_{1/2}}|F|^pdx\Big)^{\frac{1}{p}}
\bigg\},
\end{aligned}
\end{equation*}
where we choose $M_0 = \nabla u_0(0)$.
For any
$M\in \mathbb{R}^{m\times d}$, we let $\tilde{u}_0 = u_0 - Mx-\tilde{c}$.
Obviously, it satisfies the equation:
\begin{equation*}
\mathcal{L}_0(\tilde{u}_0) = F -\mathcal{L}_0(Mx+\tilde{c})
\quad\text{in}~~D_2,\qquad \frac{\partial \tilde{u}_0}{\partial\nu_0}
= g - \frac{\partial}{\partial\nu_0}\big(Mx+\tilde{c}\big)\quad \text{on}~\Delta_2.
\end{equation*}
Hence, it follows from boundary Schauder estimates (see for example \cite[Lemma 2.19]{X1}) that
\begin{equation*}
\begin{aligned}
\big\|\tilde{u}_0\big\|_{C^{1,\sigma}(D_{1/2})}
\leq CG(1,u_0).
\end{aligned}
\end{equation*}
Note that
\begin{equation*}
\begin{aligned}
\|u_0\|_{C^{1,\sigma}(D_{1/2})} &\leq \big\|\tilde{u}_0\big\|_{C^{1,\sigma}(D_{1/2})}
+ |M| + \|Mx+\tilde{c}\|_{L^\infty(D_{1/2})}\\
&\leq \big\|\tilde{u}_0\big\|_{C^{1,\sigma}(D_{1/2})}
+ |M| + C\Big(\dashint_{D_1}|Mx+\tilde{c}|^p\Big)^{\frac{1}{p}}
\end{aligned}
\end{equation*}
where we use the fact that $Mx + \tilde{c}$ is harmonic in $\mathbb{R}^d$.

It is clear to see that there exists $\theta\in(0,1/4)$ such that
$G(\theta,u_0)\leq \frac{1}{2}G(1,u_0)$.
Then the desire result $\eqref{pri:5.10}$ can be obtained simply by a rescaling argument.
\end{proof}

For simplicity, we also denote $\Phi(r)$ by
\begin{equation*}
\begin{aligned}
\Phi(r) = \frac{1}{r}\bigg\{
\Big(\dashint_{D_{r}}|u_\varepsilon - \tilde{c}|^2 \Big)^{1/2}
+ r^2\Big(\dashint_{D_{r}} |F|^p\Big)^{1/p}
+ r\|g\|_{L^\infty(\Delta_r)} + r|\tilde{c}|\bigg\}.
\end{aligned}
\end{equation*}

\begin{lemma}\label{lemma:5.6}
Let $\rho$ be given in Lemma $\ref{lemma:4.1}$.
Assume the same conditions as in Theorem $\ref{thm:4.1}$.
Let $u_\varepsilon$ be the solution of
$\mathcal{L}_\varepsilon(u_\varepsilon) = F$ in $D_2$ with
$\partial u_\varepsilon/\partial\nu_\varepsilon= g$ on $\Delta_2$.
Then we have
\begin{equation}
 G(\theta r, u_\varepsilon) \leq \frac{1}{2}G(r,u_\varepsilon)
 + C\left(\frac{\varepsilon}{r}\right)^\rho\Phi(2r)
\end{equation}
for any $r\in[\varepsilon,1/2]$, where $\theta\in(0,1/4)$ is given in Lemma $\ref{lemma:5.5}$.
\end{lemma}

\begin{proof}
Fix $r\in[\varepsilon,1/2]$, let $w$ be a solution to
$\mathcal{L}_0(w) = F$ in $D_r$,
and $\partial w/\partial\nu_0 = \partial u_\varepsilon/\partial\nu_\varepsilon$ on
$\partial D_r$. Also, let $v_\varepsilon = u_\varepsilon - \tilde{c}
- \varepsilon\chi_0(x/\varepsilon)\tilde{c}$ and $v_0 = w - \tilde{c}$.
Then we obtain
\begin{equation*}
\begin{aligned}
G(\theta r,u_\varepsilon)
&\leq \frac{1}{\theta r}\Big(\dashint_{D_{\theta r}}|u_\varepsilon - w|^2 \Big)^{\frac{1}{2}}
+ G(\theta r, w) \\
&\leq \frac{C}{r}\Big(\dashint_{D_{r}}|u_\varepsilon - w|^2\Big)^{\frac{1}{2}}
+ \frac{1}{2}G(r, w)\\
&\leq \frac{1}{2}G(r, u_\varepsilon)
+ \frac{C}{r}\Big(\dashint_{D_{r}}|u_\varepsilon - w|^2\Big)^{\frac{1}{2}}\\
&\leq \frac{1}{2}G(r, u_\varepsilon)
+ \frac{C}{r}\Big(\dashint_{D_{r}}|v_\varepsilon - v_0|^2\Big)^{\frac{1}{2}}
+ C(\varepsilon/r)|\tilde{c}|\\
&\leq  \frac{1}{2}G(r, u_\varepsilon) + C(\varepsilon/r)^\rho
 \bigg\{
 \frac{1}{r}\Big(\dashint_{D_{2r}}|u_\varepsilon-\tilde{c}|^2 \Big)^{1/2}
 + r\Big(\dashint_{D_{2r}} |F|^p\Big)^{1/p}
 + \|g\|_{L^\infty(\Delta_{2r})}+|\tilde{c}|\bigg\},
\end{aligned}
\end{equation*}
where we use the estimate $\eqref{pri:5.10}$ in the second inequality,
and $\eqref{pri:4.8}$ in the last one. The proof is complete.
\end{proof}

\begin{lemma}\label{lemma:4.2}
Let $\Psi(r)$ and $\psi(r)$ be two nonnegative continuous functions on the integral $(0,1]$.
Let $0<\varepsilon<\frac{1}{4}$. Suppose that there exists a constant $C_0$ such that
\begin{equation}\label{pri:4.2}
\left\{\begin{aligned}
  &\max_{r\leq t\leq 2r} \Psi(t) \leq C_0 \Psi(2r),\\
  &\max_{r\leq s,t\leq 2r} |\psi(t)-\psi(s)|\leq C_0 \Psi(2r),
  \end{aligned}\right.
\end{equation}
and $0\leq c(2r) \leq C_0 c(1)$ for any $r\in[\varepsilon,1/2]$.  We further assume that
\begin{equation}\label{pri:4.3}
\Psi(\theta r)\leq \frac{1}{2}\Psi(r) + C_0w(\varepsilon/r)\Big\{\Psi(2r)+\psi(2r)+c(2r)\Big\}
\end{equation}
holds for any $r\in[\varepsilon,1/2]$, where $\theta\in(0,1/4)$ and $w$ is a nonnegative
increasing function in $[0,1]$ such that $w(0)=0$ and
\begin{equation}\label{pri:4.5}
 \int_0^1 \frac{w(t)}{t} dt <\infty.
\end{equation}
Then, we have
\begin{equation}\label{pri:4.4}
\max_{\varepsilon\leq r\leq 1}\Big\{\Psi(r)+\psi(r)\Big\}
\leq C\Big\{\Psi(1)+\psi(1)+c(1)\Big\},
\end{equation}
where $C$ depends only on $C_0, \theta$ and $w$.
\end{lemma}

\begin{proof}
Here we refer the reader to \cite[Lemma 8.5]{S5}. Although
we make a few modification on it, the proof is almost the same thing.
\end{proof}

\noindent\textbf{Proof of Theorem $\ref{thm:4.1}$}.
It is fine to assume $0<\varepsilon<1/4$, otherwise it follows from the classical theory.
In view of Lemma $\ref{lemma:4.2}$,
we set $\Psi(r) = G(r,u_\varepsilon)$, $w(t) =t^\lambda$,
where $\lambda>0$ is given in Lemma $\ref{lemma:4.1}$. In order to prove the desired estimate
$\eqref{pri:4.4}$, it is sufficient to verify $\eqref{pri:4.2}$ and $\eqref{pri:4.3}$.
Let $\psi(r) = |M_r|$, where $M_r$ is the matrix associated with $\Psi(r)$, respectively.
\begin{equation*}
\begin{aligned}
\Psi(r) &= \frac{1}{r}
\Bigg\{\Big(\dashint_{D_r}|u_\varepsilon-M_r x - \tilde{c}|^2\Big)^{\frac{1}{2}}
+ r^2\Big(\dashint_{D_r}|F|^p\Big)^{\frac{1}{p}}+r^2|M_r|
+ r^2\Big(\dashint_{D_r}|M_rx+\tilde{c}|^p\Big)^{\frac{1}{p}}\\
&\qquad\quad + r\Big\|g-\frac{\partial}{\partial\nu_0}\big(M_rx+\tilde{c}\big)\Big\|_{L^\infty(\Delta_{r})}
+r^{1+\sigma}\Big[g-\frac{\partial}{\partial\nu_0}\big(M_r x + \tilde{c}\big)\Big]_{C^{0,\sigma}(\Delta_r)}\Bigg\},
\end{aligned}
\end{equation*}
Then we have
\begin{equation*}
 \Phi(r) \leq C\Big\{\Psi(2r) + \psi(2r) + c(2r)\Big\},
\end{equation*}
where $c(2r) = \|g\|_{L^\infty(\Delta_{2r})} + |\tilde{c}|$.
This together with Lemma $\ref{lemma:5.6}$ gives
\begin{equation*}
\Psi(\theta r)\leq \frac{1}{2}\Psi(r) + C_0 w(\varepsilon/r)\Big\{\Psi(2r)+\psi(2r)+c(2r)\Big\},
\end{equation*}
which satisfies the condition $\eqref{pri:4.3}$ in Lemma $\ref{lemma:4.2}$.
Let $t,s\in [r,2r]$, and $v(x)=(M_t-M_s)x$. It is clear to see $v$ is harmonic in $\mathbb{R}^d$.
Since $D_r$ satisfies the interior ball condition, we arrive at
\begin{equation}\label{f:5.15}
\begin{aligned}
|M_t-M_s|&\leq \frac{C}{r}\Big(\dashint_{D_r}|(M_t-M_s)x-\tilde{c}|^2\Big)^{\frac{1}{2}}\\
&\leq \frac{C}{t}
\Big(\dashint_{D_t}|u_\varepsilon - M_tx-\tilde{c}|^2\Big)^{\frac{1}{2}}
+ \frac{C}{s}\Big(\dashint_{D_s}|u_\varepsilon - M_sx-\tilde{c}|^2\Big)^{\frac{1}{2}}\\
&\leq C\Big\{\Psi(t)+\Psi(s)\Big\}\leq C\Psi(2r),
\end{aligned}
\end{equation}
where the second and the last steps are based on the fact that $s,t\in[r,2r]$. Due to the same reason, it
is easy to obtain $\Psi(r)\leq C\Psi(2r)$, where we use the assumption $p>d$.
The estimate $\eqref{f:5.15}$ satisfies the condition
$\eqref{pri:4.2}$. Besides, $w$ here obviously satisfies the condition $\eqref{pri:4.5}$.
Hence, according to Lemma $\ref{lemma:4.2}$, for any $r\in[\varepsilon,1/4]$,
we have the following estimate
\begin{equation}\label{f:4.6}
\frac{1}{r}\Big(\dashint_{D_{2r}}|u_\varepsilon - \tilde{c}|^2 \Big)^{\frac{1}{2}}
\leq C\Big\{\Psi(2r) + \psi(2r)\Big\}
\leq C\Big\{\Psi(1) + \psi(1) + c(1)\Big\}.
\end{equation}
Hence, for $\varepsilon\leq r<(1/4)$, the desired estimate $\eqref{pri:4.6}$ consequently follows from
$\eqref{f:4.6}$ and Caccioppoli's inequality $\eqref{pri:4.7}$,
\begin{equation*}
\begin{aligned}
\Big(\dashint_{D_r}|\nabla u_\varepsilon|^2\Big)^{1/2}
&\leq \Big(\dashint_{D_r}|\nabla v_\varepsilon|^2\Big)^{1/2} + C|\tilde{c}| \\
&\leq C\bigg\{\frac{1}{r}\Big(\dashint_{D_{2r}}|u_\varepsilon - \tilde{c}|^2\Big)^{1/2}
+ r\Big(\dashint_{D_{2r}}|F|^p\Big)^{1/p}
+ \|g\|_{L^\infty(\Delta_{2r})} + |\tilde{c}|\bigg\} \\
&\leq C\bigg\{\Big(\dashint_{D_{1}}|u_\varepsilon|^2\Big)^{1/2}
+ \Big(\dashint_{D_{2r}}|u_\varepsilon|^2\Big)^{1/2}
+ \Big(\dashint_{D_{1}}|F|^p\Big)^{1/p}
+ \|g\|_{C^{0,\sigma}(\Delta_{1})}\bigg\},
\end{aligned}
\end{equation*}
where $v_\varepsilon = u_\varepsilon - \tilde{c} - \varepsilon\chi_0(x/\varepsilon)\tilde{c}$,
and we also use the following estimate
\begin{equation}
\begin{aligned}
|\tilde{c}| = |u_0(0)| \leq C\Big(\dashint_{D_r}|u_0|^2\Big)^{1/2}
&\leq C\Big(\dashint_{D_r}|u_\varepsilon|^2\Big)^{1/2}
+ C\Big(\dashint_{D_r}|u_\varepsilon - u_0|^2\Big)^{1/2} \\
&\leq C\bigg\{\Big(\dashint_{D_{2r}}|u_\varepsilon|^2\Big)^{1/2}
+ \Big(\dashint_{D_{1}}|F|^p\Big)^{1/p}
+ \|g\|_{L^{\infty}(\Delta_{1})}\bigg\}
\end{aligned}
\end{equation}
in the last step, which is due to the estimate $\eqref{pri:4.1}$ and the fact $r\geq\varepsilon$.
We have completed the proof.
\qed

\noindent\textbf{Proof of Theorem $\ref{thm:1.0}$}.
By a rescaling argument we may prove $\eqref{pri:1.0}$ for $r=1$.
Let $u_\varepsilon = v_\varepsilon + w_\varepsilon$, where $v_\varepsilon,w_\varepsilon$ satisfy
\begin{equation*}
(1)\left\{\begin{aligned}
\mathcal{L}_\varepsilon(v_\varepsilon) &= F &\text{in}&~D_1,\\
\frac{\partial v_\varepsilon}{\partial\nu_\varepsilon} &= g
&\text{on}&~\Delta_1,
\end{aligned}\right.
\qquad
(2)\left\{\begin{aligned}
\mathcal{L}_\varepsilon(w_\varepsilon) &= \text{div}(f) &\text{in}&~D_1,\\
\frac{\partial w_\varepsilon}{\partial\nu_\varepsilon} &= -n\cdot f
&\text{on}&~\partial D_1,
\end{aligned}\right.
\end{equation*}
respectively. For (1), we claim that we can prove
\begin{equation}\label{f:4.30}
\|\nabla v_\varepsilon\|_{L^\infty(D_{1/2})}
\leq \bigg\{\|v_\varepsilon\|_{L^2(D_{1})}+ \|F\|_{L^p(D_1)} + \|g\|_{C^{0,\sigma}(\Delta_1)}\bigg\},
\end{equation}
where $C$ depends on $\mu,\tau,\kappa,\lambda,p,\sigma$ and the character of $\Omega$.
In terms of $(2)$, it follows the global Lipschitz estimate \cite[Theorem 1.2]{X2} that
\begin{equation}\label{f:4.31}
\|\nabla w_\varepsilon\|_{L^\infty(D_{1})}
\leq C\|f\|_{C^{0,\sigma}(D_1)}.
\end{equation}
Combining the estimates $\eqref{f:4.30}$ and $\eqref{f:4.31}$ lead to
the stated estimate $\eqref{pri:1.0}$,
in which we also need $H^1$ estimate for $w_\varepsilon$ (see for example \cite[Lemma 3.1]{X2}).

We now turn to prove the estimate $\eqref{f:4.30}$.
Let $\Delta_{1/2}\subset\cup_{i=1}^{N_0} B(Q_i,r)\subset \Delta_{2/3}$ for $Q_i\in\partial\Omega$ and
some $0<r<1$.
Let $\tilde{v}_\varepsilon = v_\varepsilon - \xi -\varepsilon\chi_0(x/\varepsilon)\xi$,
and $\tilde{v}_\varepsilon$ satisfies the equation $\eqref{eq:3.1}$ in $D(Q_i,r)$.
By translation we may assume $Q_i = 0$. Then it follows
classical boundary Lipschitz estimate (see \cite[Lemma 2.19]{X1}) that
\begin{equation}
\begin{aligned}
\|\nabla \tilde{v}_\varepsilon\|_{L^\infty(D_{\varepsilon/2})}
&\leq C\bigg\{\frac{1}{\varepsilon}\Big(\dashint_{D_\varepsilon}|v_\varepsilon
- \xi|^2\Big)^{1/2}
+ |\xi| + \mathcal{R}(F,0,g;\varepsilon)\bigg\} \\
&\leq C\bigg\{\Big(\dashint_{D_\varepsilon}|\nabla v_\varepsilon|^2\Big)^{1/2}
+ \Big(\dashint_{D_\varepsilon}|v_\varepsilon|^2\Big)^{1/2} + \mathcal{R}(F,0,g;1)\bigg\} \\
&\leq C\bigg\{\Big(\dashint_{D_{1}}|v_\varepsilon|^2\Big)^{1/2}
 + \mathcal{R}(F,0,g;1)\bigg\},
\end{aligned}
\end{equation}
where we choose $\xi = \dashint_{D_\varepsilon} v_\varepsilon$ in the first line, and the
second step follows from Poincar\'e's inequality and the fact $p>d$. In the last one,
the estimate $\eqref{pri:4.6}$ and the uniform H\"older estimate $\eqref{pri:2.3}$ have been employed,
and by a simple covering argument we have proved the stated estimate $\eqref{f:4.30}$,
and completed the whole proof.
\qed

\section{Neumann functions}

Let $\mathbf{\Gamma}_\varepsilon(x,y)$ denote the matrix of fundamental solutions of
$\mathcal{L}_\varepsilon$ in $\mathbb{R}^d$, with pole at $y$.
Suppose that the coefficients of $\mathcal{L}_\varepsilon$ satisfy
$\eqref{a:1}$, $\eqref{a:2}$, $\eqref{a:3}$ and $\eqref{a:4}$ with $\mu\geq\max\{\mu,\lambda_0\}$,
one may use \cite[Theorem 1.1]{X2} to show that for $d\geq 3$,
\begin{equation}\label{pri:5.2}
\begin{aligned}
\big|\mathbf{\Gamma}_\varepsilon(x,y)\big|&\leq C|x-y|^{2-d} \\
\big|\nabla_x\mathbf{\Gamma}_\varepsilon(x,y)\big|
+ \big|\nabla_y\mathbf{\Gamma}_\varepsilon(x,y)\big| &\leq C|x-y|^{1-d},
\end{aligned}
\end{equation}
where $C$ depends only on $\mu,\kappa,\lambda,\tau,m,d$. Let $U_\varepsilon(x,y)$ be the solution of
\begin{equation}\label{pde:5.1}
\left\{\begin{aligned}
\mathcal{L}_\varepsilon(U_\varepsilon^\beta(\cdot,y)) &= 0 &\quad&\text{in}~~\Omega,\\
\frac{\partial}{\partial\nu_\varepsilon}\big(U_\varepsilon^\beta(\cdot,y)\big)
& = \frac{\partial}{\partial\nu_\varepsilon}\big(\mathbf{\Gamma}_\varepsilon^\beta(\cdot,y)\big)
&\quad&\text{on}~\partial\Omega,
\end{aligned}
\right.
\end{equation}
where $\mathbf{\Gamma}_\varepsilon^\beta(x,y)
= \big(\mathbf{\Gamma}_\varepsilon^{1\beta}(x,y),\cdots,\mathbf{\Gamma}_\varepsilon^{m\beta}(x,y)\big)$.
We now define
\begin{equation}
\mathbf{N}_\varepsilon(x,y) = \mathbf{\Gamma}_\varepsilon(x,y) - U_\varepsilon(x,y)
\end{equation}
for $x,y\in\Omega$. Note that, if
$\mathbf{N}_\varepsilon^\beta(x,y) = \mathbf{\Gamma}_\varepsilon^\beta(x,y) - U_\varepsilon^\beta(x,y)$,
\begin{equation}
\left\{\begin{aligned}
\mathcal{L}_\varepsilon\big(\mathbf{N}_\varepsilon^\beta(\cdot,y)\big) &= e^\beta\delta_y
&\quad&\text{in}~~\Omega,\\
\frac{\partial}{\partial\nu_\varepsilon}\big(\mathbf{N}_\varepsilon^\beta(\cdot,y)\big)
& = 0 &\quad&\text{on}~\partial\Omega,
\end{aligned}\right.
\end{equation}
where $\delta_y$ denotes the Dirac delta function with pole at $y$.
We will call $\mathbf{N}_\varepsilon(x,y)$ the matrix of Neumann functions for
$\mathcal{L}_\varepsilon$ in $\Omega$.

\begin{lemma}
Suppose that the coefficients of $\mathcal{L}_\varepsilon$ satisfy $\eqref{a:1}$,
$\eqref{a:2}$, $\eqref{a:3}$ and $\eqref{a:4}$ with $\lambda\geq\max\{\mu,\lambda_0\}$.
Let $U_\varepsilon(x,y)$ be defined by $\eqref{pri:5.1}$. Then there holds
\begin{equation}\label{pri:5.5}
|U_\varepsilon(x,y)|\leq C\big[\delta(x)\big]^{\frac{2-d}{2}}\big[\delta(y)\big]^{\frac{2-d}{2}}
\end{equation}
for any $x,y\in\Omega$, where $\delta(x)= \emph{dist}(x,\partial\Omega)$ and $C$ depends on
$\mu,d,m$ and $\Omega$.
\end{lemma}

\begin{proof}
Fix $y\in\Omega$, and let $w_\varepsilon(x) = U_\varepsilon(x,y)$. In view of $\eqref{pde:5.1}$ we have
\begin{equation}\label{f:5.1}
\|w_\varepsilon\|_{H^1(\Omega)}
\leq C\big\|\frac{\partial w_\varepsilon}{\partial\nu_\varepsilon}\big\|_{H^{-\frac{1}{2}}(\partial\Omega)}
\leq C\big\|\frac{\partial w_\varepsilon}{\partial\nu_\varepsilon}\big\|_{L^{p}(\partial\Omega)},
\end{equation}
where $p= \frac{2(d-1)}{d}$. On account of the estimates $\eqref{pri:5.2}$,
\begin{equation}\label{f:5.2}
\begin{aligned}
\big\|\frac{\partial w_\varepsilon}{\partial\nu_\varepsilon}\big\|_{L^{p}(\partial\Omega)}
&\leq C\Big\{\int_{\partial\Omega}\frac{dS(x)}{|x-y|^{p(d-1)}}\Big\}^{1/p} \\
&\leq C\bigg\{\int_{\delta(y)}^{\infty}s^{p(1-d)+d-2}ds
+ \int_{B(y,2\delta(y))\cap\partial\Omega}\frac{dS(y)}{|x-y|^{p(d-1)}}\bigg\}^{1/p}
\leq C\big[\delta(y)\big]^{\frac{2-d}{2}}.
\end{aligned}
\end{equation}

Also, it follows from interior estimate \cite[Corollary 3.5]{X0} that
\begin{equation*}
\begin{aligned}
|w_\varepsilon(x)|&\leq C\Big(\dashint_{B(x,\delta(x))}|w_\varepsilon|^{2^*}\Big)^{1/2^*}
\leq C\big[\delta(x)\big]^{\frac{2-d}{d}}\big\|w_\varepsilon\big\|_{H^1(\Omega)}
\leq C\big[\delta(x)\big]^{\frac{2-d}{d}}\big[\delta(y)\big]^{\frac{2-d}{d}},
\end{aligned}
\end{equation*}
where we use the estimates $\eqref{f:5.1}$ and $\eqref{f:5.2}$ in the last inequality,
and this ends the proof.
\end{proof}

\begin{thm}
Let $\Omega\subset\mathbb{R}^d$ be a bounded $C^{1,\tau}$ domain.
Suppose that the coefficients of $\mathcal{L}_\varepsilon$ satisfy $\eqref{a:1}$,
$\eqref{a:2}$, $\eqref{a:3}$ and $\eqref{a:4}$ with $\lambda\geq\max\{\mu,\lambda_0\}$. Then
\begin{equation}\label{pri:5.3}
\big|\mathbf{N}_\varepsilon(x,y)\big|\leq C|x-y|^{2-d}
\end{equation}
and for any $\sigma\in(0,1)$,
\begin{equation}\label{pri:5.4}
\begin{aligned}
\big|\mathbf{N}_\varepsilon(x,y)-\mathbf{N}_\varepsilon(z,y)\big|
&\leq C_\sigma\frac{|x-z|^\sigma}{|x-y|^{d-2+\sigma}},\\
\big|\mathbf{N}_\varepsilon(y,x)-\mathbf{N}_\varepsilon(y,z)\big|
&\leq C_\sigma\frac{|x-z|^\sigma}{|x-y|^{d-2+\sigma}},
\end{aligned}
\end{equation}
where $|x-z|<|x-y|/4$.
\end{thm}

\begin{proof}
Due to the boundary H\"older's estimate $\eqref{pri:2.12}$, it suffices to prove
the estimate $\eqref{pri:5.3}$. By the estimate $\eqref{pri:5.5}$,
\begin{equation}\label{f:3.5}
\big|\mathbf{N}_\varepsilon(x,y)\big|\leq C\Big\{|x-y|^{2-d}
+\big[\delta(x)\big]^{2-d}+\big[\delta(y)\big]^{2-d}\Big\}.
\end{equation}
Then let $r=|x-y|$, and it follows from the estimates $\eqref{pri:2.3}$, $\eqref{f:3.5}$ that
\begin{equation*}
\begin{aligned}
\big|\mathbf{N}_\varepsilon(x,y)\big|
&\leq C\bigg\{\dashint_{B(x,r/4)\cap\Omega}|\mathbf{N}_\varepsilon(z,y)|^sdz\bigg\}^{1/s} \\
&\leq C\Big\{|x-y|^{2-d}+[\delta(y)]^{2-d}\Big\},
\end{aligned}
\end{equation*}
where we choose $s>0$ such that $s(d-2)<1$. Using the estimate $\eqref{pri:2.3}$ again, the
above estimate leads to
\begin{equation*}
\big|\mathbf{N}_\varepsilon(x,y)\big|
\leq C\bigg\{\dashint_{B(y,r/4)\cap\Omega}|\mathbf{N}_\varepsilon(x,z)|^sdz\bigg\}^{1/s}
\leq C|x-y|^{2-d},
\end{equation*}
and we have completed the proof.
\end{proof}

\begin{thm}
Let $\Omega\subset\mathbb{R}^d$ be a bounded $C^{1,\tau}$ domain. Suppose that
the coefficients of $\mathcal{L}_\varepsilon$ satisfy $\eqref{a:1}$, $\eqref{a:2}$,
$\eqref{a:3}$ and $\eqref{a:4}$ with $\lambda\geq\max\{\mu,\lambda_0\}$.
Let $x_0,y_0,z_0\in\Omega$ be such that $|x_0-z_0|<|x_0-y_0|/4$. Then for any $\sigma\in(0,1)$,
\begin{equation}\label{pri:5.1}
\bigg(\dashint_{B(y_0,\rho/4)\cap\Omega}
\big|\nabla_y\big\{\mathbf{N}_\varepsilon(x_0,y)-\mathbf{N}_\varepsilon(z_0,y)\big\}\big|^2dy\bigg)^{1/2}
\leq C\rho^{1-d}\Big(\frac{|x_0-z_0|}{\rho}\Big)^{\sigma},
\end{equation}
where $\rho = |x_0-y_0|$ and $C$ depends only on $\mu,\kappa,\lambda,\tau,m,d,\sigma$
and the character of $\Omega$.
\end{thm}

\begin{proof}
Let $f\in C_0^\infty(B(y_0,\rho/2)\cap\Omega)$, and
\begin{equation*}
 u_\varepsilon(x) = \int_{\Omega}\mathbf{N}_\varepsilon(x,y)f(y)dy.
\end{equation*}
Then $\mathcal{L}_\varepsilon(u_\varepsilon) = f$ in $\Omega$ and
$\partial u_\varepsilon/\partial\nu_\varepsilon = 0$ in $B(x_0,\rho/2)\cap\partial\Omega$, it
follows from the boundary H\"older estimate and interior estimates
\begin{equation}
\big|u_\varepsilon(x_0)-u_\varepsilon(z_0)\big|
\leq C\Big(\frac{|x_0-z_0|}{\rho}\Big)^{\sigma}
\bigg\{\dashint_{B(x_0,\rho/2)\cap\Omega}|u_\varepsilon|^2\bigg\}^{1/2}.
\end{equation}
On the other hand, for any $x\in B(x_0,\rho/2)\cap\Omega$, we obtain
\begin{equation}
 \big|u_\varepsilon(x)\big|\leq C\rho^2 \bigg\{\dashint_{B(y_0,\rho/2)\cap\Omega}|f|^2 dy\bigg\}^{1/2},
\end{equation}
and this implies
\begin{equation}
\Big\{\dashint_{B(x_0,\rho/2)\cap\Omega}|u_\varepsilon|^2 dx\Big\}^{1/2}
\leq C\rho^{2-d/2}\big\|f\big\|_{L^2(\Omega)}.
\end{equation}

Thus we obtain that
\begin{equation}\label{f:5.3}
\bigg(\dashint_{B(y_0,\rho/2)\cap\Omega}
\big|\mathbf{N}_\varepsilon(x_0,y)-\mathbf{N}_\varepsilon(z_0,y)\big|^2dy\bigg)^{1/2}
\leq C\rho^{2-d}\Big(\frac{|x_0-z_0|}{\rho}\Big)^{\sigma},
\end{equation}
and the stated estimate $\eqref{pri:5.1}$ follows from Caccippoli's inequality $\eqref{pri:2.12}$. We
have completed the proof.
\end{proof}

\section{The proof of Theorem $\ref{thm:1.1}$}

In the case of $p=2$, the estimate $\eqref{pri:1.1}$ has been established
in \cite[Theorem 1.6]{X2}. For $2<p<\infty$, according to Theorem $\ref{thm:2.1}$, it suffices to
establish the following reverse H\"older inequality.

\begin{lemma}\label{lemma:6.1}
Let $\Omega$ be a bounded $C^{1,\eta}$ domain. Suppose the coefficients of $\mathcal{L}_\varepsilon$
satisfy $\eqref{a:1}$, $\eqref{a:2}$, $\eqref{a:3}$ and $\eqref{a:4}$ with $\lambda\geq\lambda_0$ and
$A=A^*$.
For any $Q\in\partial\Omega$ and $0<r<1$,
let $u_\varepsilon\in H^1(B(Q,3r)\cap\partial\Omega;\mathbb{R}^m)$ be the weak solution of
$\mathcal{L}_\varepsilon(u_\varepsilon) = 0$ in $B(Q,3r)\cap\Omega$, and
$\partial u_\varepsilon/\partial\nu_\varepsilon = 0$ on $B(Q,3r)\cap\partial\Omega$.
Then we have
\begin{equation}\label{pri:6.1}
\sup_{B(Q,r)\cap\partial\Omega}|(\nabla u_\varepsilon)^*|
\leq C\bigg\{\Big(\dashint_{B(Q,2r)\cap\partial\Omega}|(\nabla u_\varepsilon)^*|^2\Big)^{1/2}
+\Big(\dashint_{B(Q,2r)\cap\partial\Omega}|(u_\varepsilon)^*|^2\Big)^{1/2}\bigg\},
\end{equation}
where $C$ depends on $\mu,\tau,\kappa,\lambda,m,d$ and the character of $\Omega$.
\end{lemma}

\begin{proof}
The main idea could be found in \cite[Lemma 9.1]{KFS1}, and we have to impose some
new tricks to derive $\eqref{pri:6.1}$, which additionally involves the so-called Neuamnn correctors defined
in \cite{X1}, i.e.,
\begin{equation*}
-\text{div}\big[A(x/\varepsilon)\nabla \Psi_{\varepsilon,0}\big] = \text{div}(V(x/\varepsilon))
\quad \text{in}~\Omega,
\qquad n\cdot A(x/\varepsilon)\nabla\Psi_{\varepsilon,0} = n\cdot{\widehat{V}-V(x/\varepsilon)}
\quad \text{on}~\partial\Omega.
\end{equation*}

The purpose is to establish the following boundary estimate
\begin{equation}\label{pri:6.2}
 \|\nabla u_\varepsilon\|_{L^\infty(D_{r/2})}
 \leq C\bigg\{\Big(\dashint_{D_r}|\nabla u_\varepsilon|^2\Big)^{1/2}
+\Big(\dashint_{D_r}|u_\varepsilon|^2\Big)^{1/2}\bigg\},
\end{equation}
where $C$ depends on $\mu,\tau,\kappa,\lambda,m,d$ and $\Omega$.

First of all, we consider $\varepsilon\leq r<1$.
Let $v_\varepsilon = u_\varepsilon - \Psi_{\varepsilon,0}\xi$ for some $\xi\in\mathbb{R}^m$, and then we have
\begin{equation}\label{pde:6.1}
\left\{\begin{aligned}
\mathcal{L}_\varepsilon(v_\varepsilon) &= \text{div}\big[V_\varepsilon(\Psi_{\varepsilon,0}-I)\xi\big]
-B_\varepsilon\nabla\Psi_{\varepsilon,0}\xi - (c_\varepsilon+\lambda I)\Psi_{\varepsilon,0}\xi
\quad&\text{in}& ~ B(Q,3r)\cap\Omega,\\
\frac{\partial v_\varepsilon}{\partial\nu_\varepsilon}
& = n\cdot V_\varepsilon(I-\Psi_{\varepsilon,0})\xi - n\cdot\widehat{V}\xi
\quad &\text{on}&~B(Q,3r)\cap\partial\Omega,
\end{aligned}\right.
\end{equation}
where $V_\varepsilon(x) = V(x/\varepsilon)$ and $c_\varepsilon(x)=c(x/\varepsilon)$. Note that
\begin{equation}\label{f:6.1}
\|\Psi_{\varepsilon,0}-I\|_{L^\infty(\Omega)}\leq C\varepsilon,
\qquad \|\Psi_{\varepsilon,0}\|_{W^{1,\infty}(\Omega)}\leq C,
\end{equation}
where $C$ depends on $\mu,\kappa,\tau,m,d$ and $\Omega$. The above results have been proved in
\cite[Theorem 4.2]{X1} and Remark $\ref{re:6.1}$.

Applying the boundary estimate $\eqref{pri:1.0}$ to the equation $\eqref{pde:6.1}$, we have
\begin{equation*}
\begin{aligned}
\|\nabla v_\varepsilon\|_{L^\infty(D_{r/2})}
&\leq C\bigg\{\frac{1}{r}\Big(\dashint_{D_r}|v_\varepsilon|^2\Big)^{1/2}
+|\xi|\bigg\} \\
&\leq C\bigg\{\frac{1}{r}\Big(\dashint_{D_r}|u_\varepsilon-\xi|^2\Big)^{1/2}
+\frac{\varepsilon}{r}|\xi|+|\xi|\bigg\}\\
&\leq C\bigg\{\frac{1}{r}\Big(\dashint_{D_r}|u_\varepsilon-\xi|^2\Big)^{1/2}+|\xi|\bigg\}\\
&\leq C\bigg\{\Big(\dashint_{D_r}|\nabla u_\varepsilon|^2\Big)^{1/2}
+\Big(\dashint_{D_r}|u_\varepsilon|^2\Big)^{1/2}\bigg\}
\end{aligned}
\end{equation*}
where we use the estimate $\eqref{f:6.1}$ in the first and second inequalities, and
the fact $r\geq \varepsilon$ in the third one. In the last step, we choose
$\xi=\dashint_{D_r} u_\varepsilon$ and employ Poincar\'e's inequality.
The above estimate implies $\eqref{pri:6.2}$ for $\varepsilon\leq r<1$.

In the case of $0<r<\varepsilon$,
let $v_\varepsilon = u_\varepsilon - \xi - \varepsilon\chi_0(x/\varepsilon)\xi$
for some $\xi\in\mathbb{R}^m$, and $v_\varepsilon$ satisfies the equation
$\eqref{eq:3.1}$ by setting $F=g=0$. Again, by choosing $\xi=\dashint_{D_r} u_\varepsilon$,
and the estimate $\eqref{pri:1.0}$, one may derive the estimate $\eqref{pri:6.2}$
for $0<r<\varepsilon$.

Recall the definition of nontangential maximal function, and we have
\begin{equation*}
 (\nabla u_\varepsilon)^*(P) = \max\{\mathcal{M}_{r,1}(\nabla u_\varepsilon)(P),
 \mathcal{M}_{r,2}(\nabla u_\varepsilon)(P)\},
\end{equation*}
where
\begin{equation}\label{def:6.1}
\begin{aligned}
\mathcal{M}_{r,1}(\nabla u_\varepsilon)(P)
&=\big\{|\nabla u_\varepsilon|:x\in\Omega, |x-P|\leq c_0r;
~|x-P|\leq N_0\text{dist}(x,\partial\Omega)\big\},\\
\mathcal{M}_{r,2}(\nabla u_\varepsilon)(P)
&=\big\{|\nabla u_\varepsilon|:x\in\Omega, |x-P|> c_0r;
~|x-P|\leq N_0\text{dist}(x,\partial\Omega)\big\}.
\end{aligned}
\end{equation}
Here $c_0>0$ is a small constant.
We first handle the estimate for $\mathcal{M}_{r,1}(\nabla u_\varepsilon)$.
It follows from the estimate $\eqref{pri:6.2}$ that
\begin{equation}\label{f:6.2}
\begin{aligned}
\sup_{B(Q,r)\cap\partial\Omega}\mathcal{M}_{r,1}(\nabla u_\varepsilon)
&\leq \sup_{B(Q,3r/2)\cap\partial\Omega}|\nabla u_\varepsilon| \\
&\leq C\bigg\{\Big(\dashint_{D_{2r}}|\nabla u_\varepsilon|^2\Big)^{1/2}
+\Big(\dashint_{D_{2r}}|u_\varepsilon|^2\Big)^{1/2}\bigg\} \\
&\leq C\bigg\{\Big(\dashint_{\Delta_{2r}}|(\nabla u_\varepsilon)^*|^2\Big)^{1/2}
+\Big(\dashint_{\Delta_{2r}}|(u_\varepsilon)^*|^2\Big)^{1/2}\bigg\}.
\end{aligned}
\end{equation}

Similarly, one may derive the following interior Lipschitz estimate
\begin{equation*}
 |\nabla u_\varepsilon(x)|\leq
 C\bigg\{\Big(\dashint_{B(x,\delta(x)/10)}|\nabla u_\varepsilon|^2\Big)^{1/2}
 +\Big(\dashint_{B(x,\delta(x)/10)}|u_\varepsilon|^2\Big)^{1/2}\bigg\}
\end{equation*}
from \cite[Theorem 4.4]{X0}, and this will give
\begin{equation}\label{f:6.3}
\begin{aligned}
\sup_{B(Q,r)\cap\partial\Omega}\mathcal{M}_{r,2}(\nabla u_\varepsilon)
\leq C\bigg\{\Big(\dashint_{\Delta_{2r}}|(\nabla u_\varepsilon)^*|^2\Big)^{1/2}
+\Big(\dashint_{\Delta_{2r}}|(u_\varepsilon)^*|^2\Big)^{1/2}\bigg\}.
\end{aligned}
\end{equation}
Combining the estimates $\eqref{f:6.2}$ and $\eqref{f:6.3}$ consequently lead to
the stated estimate $\eqref{pri:6.1}$. We have completed the proof.
\end{proof}

\begin{remark}\label{re:6.1}
\emph{Here we plan give a sketch of the proof of the first estimate in $\eqref{f:6.1}$.
Let $H_{\varepsilon,0} = \Psi_{\varepsilon,0} - I -\varepsilon\chi_0(x/\varepsilon)$. Then it satisfies
$L_\varepsilon(H_{\varepsilon,0}) = 0$ in $\Omega$ and
$\partial H_{\varepsilon,0}/\partial n_\varepsilon = \sum_{ij}
\big(n_i\frac{\partial}{\partial x_j}-n_j\frac{\partial}{\partial x_i}\big)g_{ij}$ on $\partial\Omega$
with $\|g_{ij}\|_{L^\infty(\Omega)}\leq C\varepsilon$ (the computation will be found in
\cite[Lemma 4.4]{X1}),
where $\partial/\partial n_\varepsilon = n\cdot A(x/\varepsilon)\nabla$ denotes
the conormal derivative operator associated with $L_\varepsilon$. According to the proof of \cite[Lemma 4.3]{X1}, one may have
\begin{equation*}
|H_{\varepsilon,0}(x) - H_{\varepsilon,0}(y)|\leq C\varepsilon
\end{equation*}
for any $x,y\in\Omega$. Thus, it is not hard to see that
\begin{equation*}
\|H_{\varepsilon,0}(x)\|_{L^\infty(\Omega)}\leq C\varepsilon + C\|H_{\varepsilon,0}\|_{L^2(\Omega)}.
\end{equation*}
Note that
\begin{equation*}
\|H_{\varepsilon,0}\|_{L^2(\Omega)}
\leq C\|(H_{\varepsilon,0})^*\|_{L^2(\partial\Omega)}
\leq C\|H_{\varepsilon,0}\|_{L^2(\partial\Omega)} \leq C\varepsilon,
\end{equation*}
and the last inequality is due to a duality argument (see \cite{KFS2}).
Let $\phi_\varepsilon\in H^1(\Omega;\mathbb{R}^m)$ be a solution of
$L_\varepsilon(\phi_\varepsilon) = 0$ in $\Omega$, and
$\partial\phi_\varepsilon/\partial n_\varepsilon = f$ on $\partial\Omega$ with
$\int_{\partial\Omega}fdS = 0$.
\begin{equation*}
\Big|\int_{\partial\Omega}H_{\varepsilon,0}f dS\Big|
=\Big|\int_{\partial\Omega} g_{ij}\Big(n_i\frac{\partial}{\partial x_j}
-n_j\frac{\partial}{\partial x_i}\Big)\phi_\varepsilon dS\Big|
\leq C\varepsilon\|f\|_{L^2(\partial\Omega)},
\end{equation*}
where we use the Rellich estimate
$\|\nabla_{\text{tan}}\phi_\varepsilon\|_{L^2(\partial\Omega)}\leq C\|f\|_{L^2(\partial\Omega)}$
(see \cite{KS1}). Thus the above estimates imply
\begin{equation*}
  \|\Psi_{\varepsilon,0} - I\|_{L^\infty(\Omega)}
  \leq \|H_{\varepsilon,0}\|_{L^\infty(\Omega)} + C\varepsilon \leq C\varepsilon.
\end{equation*}
We mention that the above estimate additionally relies on the symmetry condition $A=A^*$, which
actually improves the corresponding result in \cite[Theorem 4.2]{X1}.}
\end{remark}

One may study the solutions of the $L^2$ Neumann problem with atomic data
$\partial u_\varepsilon/\partial \nu_\varepsilon = a$ on $\partial\Omega$,
where $\int_{\partial\Omega} a(x) dS = 0$, and $\text{supp}(a)\subset B(Q,r)\cap\partial\Omega$
for some $Q\in\partial\Omega$ and $0<r<r_0$, and $\|a\|_{L^\infty(\partial\Omega)}\leq r^{1-d}$.
In fact, the stated estimate $\eqref{pri:1.1}$ holding for $1<p<2$ follows from the following
result by interpolation.

\begin{thm}\label{thm:6.1}
Let $a$ be an atom on $\Delta_r$ with $0<r<r_0$. Suppose that
$u_\varepsilon$ is a weak solution of $\mathcal{L}_\varepsilon(u_\varepsilon) = 0$ in $\Omega$
with $\partial u_\varepsilon/\partial \nu_\varepsilon = a$ on $\partial\Omega$. Then we have
the following estimate
\begin{equation}\label{pri:5.6}
\int_{\partial\Omega} (\nabla u_\varepsilon)^* dS\leq C,
\end{equation}
where $C$ depends only on $\mu,\kappa,\lambda,d,m$ and $\Omega$.
\end{thm}

\begin{proof}
The main ideas of the proof may be found in \cite[pp.932-933]{S1},
as well as in \cite[Lemma 2.7]{DK}. Clearly, the integral in the left-hand side of
$\eqref{pri:5.6}$ may be divided into
\begin{equation*}
\int_{\partial\Omega} (\nabla u_\varepsilon)^* dS
= \bigg\{
\int_{B(Q,Cr)\cap\partial\Omega} + \int_{\partial\Omega\setminus B(Q,Cr)}\bigg\}
(\nabla u_\varepsilon)^*dS,
\end{equation*}
and it follows from $L^2$ estimate (\cite[Theorem 1.6]{X2}) and H\"older's inequality that
\begin{equation}\label{f:5.6}
\int_{B(Q,Cr)\cap\partial\Omega} (\nabla u_\varepsilon)^* dS
\leq Cr^{\frac{d-1}{2}}\big\|(\nabla u_\varepsilon)^*\big\|_{L^2(\partial\Omega)}
\leq Cr^{\frac{d-1}{2}}\|a\|_{L^2(\partial\Omega)} \leq C,
\end{equation}
where we also use the assumption $\|a\|_{L^\infty(\partial\Omega)}\leq r^{1-d}$.

Let $\rho = |P_0-Q|\geq Cr$, and one may show
\begin{equation}\label{f:5.5}
\int_{B(P_0,c\rho)\cap\partial\Omega} (\nabla u_\varepsilon)^* dS
\leq C\Big(\frac{r}{\rho}\Big)^{\sigma}
\end{equation}
for some $\sigma>0$. Since $\int_{\partial\Omega} a(x)dS = 0$, we have the formula
\begin{equation*}
 u_\varepsilon(x) = \int_{B(Q,r)\cap\partial\Omega}
 \Big\{\mathbf{N}_\varepsilon(x,y)-\mathbf{N}_\varepsilon(x,Q)\Big\}a(y)dS(y),
\end{equation*}
and it follows that
\begin{equation*}
|\nabla u_\varepsilon(x)|
\leq C\dashint_{B(Q,r)\cap\partial\Omega}
\big|\nabla_x\big\{\mathbf{N}_\varepsilon(x,y)-\mathbf{N}_\varepsilon(x,Q)\big\}\big|dS(y).
\end{equation*}

Note that for any $z\in\Omega$ such that $c\rho\leq |z-P|<N_0\delta(z)$ for some
$P\in B(P_0,c\rho)\cap\partial\Omega$, and it follows from interior
Lipschitz estimates (which is based upon \cite[Theorem 4.4]{X0} coupled with the techniques in
the proof of Lemma $\ref{lemma:6.1}$), and
$\eqref{pri:5.1}$ and $\eqref{f:5.3}$ that
\begin{equation*}
\begin{aligned}
|\nabla u_\varepsilon(z)|
&\leq C \bigg(\dashint_{B(z,c\delta(z))}|\nabla u_\varepsilon|^2 dx\bigg)^{1/2}
+ C\delta(x) \dashint_{B(z,c\delta(z))} |u_\varepsilon| dx \\
&\leq C\dashint_{B(Q,r)\cap\partial\Omega}
\bigg(\dashint_{B(z,c\delta(z))}
\big|\nabla_x\big\{\mathbf{N}_\varepsilon(x,y)-\mathbf{N}_\varepsilon(x,Q)\big\}\big|^2
dx\bigg)^{1/2}dS(y) \\
& + C\delta(x)
\dashint_{B(Q,r)\cap\partial\Omega}
\bigg(\dashint_{B(z,c\delta(z))}
\big|\mathbf{N}_\varepsilon(x,y)-\mathbf{N}_\varepsilon(x,Q)\big|^2
dx\bigg)^{1/2}dS(y) \\
& \leq C\rho^{1-d}\Big(\frac{r}{\rho}\Big)^{\sigma}
+ C\delta(z)\rho^{2-d}\Big(\frac{r}{\rho}\Big)^{\sigma}
\leq C\rho^{1-d}\Big(\frac{r}{\rho}\Big)^{\sigma},
\end{aligned}
\end{equation*}
where we use
Minkowski's inequality in the second step and the fact that $(c\rho)/N_0<\delta(x)<r_0$ in the last one.
According to the definition of $\mathcal{M}_{\rho,2}$ in $\eqref{def:6.1}$, we have
\begin{equation}\label{f:5.4}
\int_{B(P_0,c\rho)\cap\partial\Omega}\mathcal{M}_{\rho,2}(\nabla u_\varepsilon) dS
\leq C\Big(\frac{r}{\rho}\Big)^{\sigma}.
\end{equation}

For any $\theta\in[1,3/2]$, it is known that
$\mathcal{L}_\varepsilon(u_\varepsilon) = 0$ in $\Omega$ and
$\partial u_\varepsilon/\partial\nu_\varepsilon = 0$ on $B(P_0,\theta c\rho)\cap\partial\Omega$.
In terms of $L^2$ nontangential maximal function estimate \cite[Thereom 1.7]{X2}, we have
\begin{equation*}
\int_{B(P_0,\theta c\rho)\cap\partial\Omega} |\mathcal{M}_{\rho,1}(\nabla u_\varepsilon)|^2 dS
\leq C\int_{\partial D_{\theta c\rho}\setminus\partial\Omega}|\nabla u_\varepsilon|^2 dS.
\end{equation*}
Integrating in $\theta$ on $[1,3/2]$ yields
\begin{equation*}
\int_{B(P_0,c\rho)\cap\partial\Omega}
|\mathcal{M}_{\rho,1}(\nabla u_\varepsilon)|^2 dS
\leq \frac{C}{\rho}
\int_{D_{2c\rho}}|\nabla u_\varepsilon|^2 dx,
\end{equation*}
and
\begin{equation*}
\int_{B(P_0,c\rho)\cap\partial\Omega}
\mathcal{M}_{\rho,1}(\nabla u_\varepsilon) dS
\leq C\rho^{d-1}\bigg(\dashint_{B(P_0,2c\rho)\cap\Omega}|\nabla u_\varepsilon|^2\bigg)^{1/2}
\leq C\Big(\frac{r}{\rho}\Big)^\sigma.
\end{equation*}
This together with $\eqref{f:5.4}$ gives the estimate $\eqref{f:5.5}$.
Consequently, the desired estimate $\eqref{pri:5.6}$ follows from
the estimates $\eqref{pri:5.5}$ and $\eqref{pri:5.6}$ by a covering argument. We are done.
\end{proof}

\noindent\textbf{Proof of Theorem $\ref{thm:1.1}$}.
The desired estimate
for $\|(\nabla u_\varepsilon)^*\|_{L^p(\partial\Omega)}$ in $\eqref{pri:1.1}$
is based upon Theorem $\ref{thm:2.1}$,
Lemma $\ref{lemma:6.1}$ and
Theorem $\ref{thm:6.1}$. In view of Lemmas $\ref{lemma:2.3}$ and $\ref{lemma:2.4}$, one may derive
\begin{equation*}
\|(u_\varepsilon)^*\|_{L^p(\partial\Omega)}
\leq C\|\mathcal{M}(u_\varepsilon)\|_{L^p(\partial\Omega)}
\leq C\|u_\varepsilon\|_{W^{1,p}(\Omega)}
\leq C\|g\|_{L^{p}(\partial\Omega)},
\end{equation*}
where we use $W^{1,p}$ estimate (see \cite[Theorem 1.1]{X1}) in the last step.
The proof is complete.
\qed

\begin{center}
\textbf{Acknowledgements}
\end{center}

The authors thank Prof. Zhongwei Shen for very helpful discussions
regarding this work when he visited Peking University.
The first author also appreciates his constant and illuminating instruction.
The first author was supported by the China Postdoctoral Science Foundation (Grant No. 2017M620490), and
the second author was supported by the National Natural Science Foundation of China
(Grant NO. 11571020).

\end{document}